\documentclass[11pt, reqno]{amsart}
\usepackage{indentfirst, amssymb, amsmath, amsthm, mathrsfs, setspace, indentfirst, enumerate,  mathrsfs, amsmath, amsthm}
\usepackage[bookmarksnumbered, colorlinks, plainpages]{hyperref}
\usepackage{mathrsfs}

\newtheorem*{theoA}{Theorem A}
\newtheorem*{theoB}{Theorem B}
\newtheorem*{theoC}{Theorem C}
\newtheorem*{theoD}{Theorem D}
\newtheorem*{theoE}{Theorem E}

\newtheorem*{conjA}{Conjecture A}

\newtheorem*{cor A}{Corollary A}
\newtheorem*{cor B}{Corollary B}

\newtheorem{theo}{Theorem}[section]
\newtheorem{lem}{Lemma}[section]
\newtheorem{cor}{Corollary}[section]

\newtheorem{exm}{Example}[section]

\newtheorem{rem}{Remark}[section]

\newcommand{\be}{\begin{equation}}
\newcommand{\ee}{\end{equation}}
\newcommand{\beas}{\begin{eqnarray*}}
\newcommand{\eeas}{\end{eqnarray*}}
\newcommand{\bea}{\begin{eqnarray}}
\newcommand{\eea}{\end{eqnarray}}

\numberwithin{equation}{section}
\begin{document}
\title[M\MakeLowercase {eromorphic solution of a certain type of algebraic....}]{\LARGE M\Large\MakeLowercase {eromorphic solution of a certain type of algebraic differential equation}}
\date{}
\author[J. F. X\MakeLowercase{u}, S. M\MakeLowercase{ajumder}, N. S\MakeLowercase{arkar} \MakeLowercase{and} L. M\MakeLowercase{ahato} ]{J\MakeLowercase{unfeng} X\MakeLowercase{u}, S\MakeLowercase{ujoy} M\MakeLowercase{ajumder}$^{*}$, N\MakeLowercase{abadwip} S\MakeLowercase{arkar} \MakeLowercase{and} L\MakeLowercase{ata} M\MakeLowercase{ahato}}
\address{Department of Mathematics, Wuyi University, Jiangmen 529020, Guangdong, People's Republic of China.}
\email{xujunf@gmail.com}
\address{Department of Mathematics, Raiganj University, Raiganj, West Bengal-733134, India.}
\email{sm05math@gmail.com, sjm@raiganjuniversity.ac.in}
\address{Department of Mathematics, Raiganj University, Raiganj, West Bengal-733134, India.}
\email{naba.iitbmath@gmail.com}
\address{Department of Mathematics, Mahadevananda Mahavidyalaya, Monirampore Barrackpore, West Bengal-700120, India.}
\email{lata27math@gmail.com}

\renewcommand{\thefootnote}{}
\footnote{*\emph{Corresponding Author}: Sujoy Majumder.}
\renewcommand{\thefootnote}{\arabic{footnote}}
\setcounter{footnote}{0}

\begin{abstract} In the paper, we use the idea of normal family to find out the possible solution of the following special case of algebraic differential equation 
\[P_k\big(z,f,f^{(1)},\ldots, f^{(k)}\big)=f^{(1)}(f-\mathscr{L}_k(f))-\varphi (f-a)(f-b)=0,\]
where $\mathscr{L}_k(f)=\sideset{}{_{i=0}^k}{\sum} a_i f^{(i)}$ and $\varphi$ is an entire function, $a_i\in\mathbb{C}\;(i=0,1,\ldots, k)$ such that $a_k=1$ and  $a, b\in\mathbb{C}$ such that $a\neq b$. The obtained results improve and generalise the results of 
Li and Yang \cite{LY1} and Xu et al. \cite{XMD} in a large scale.
\end{abstract}
\thanks{Typeset by \AmS -\LaTeX}
\maketitle

\section{{\bf Introduction and main result}}
We assume that the reader is familiar with the basic notations and main results of Nevanlinna's value distribution theory (see \cite{WKH1, YY1}). We use notation $\rho(f)$ for the order of a meromorphic function $f$. As usual, the abbreviation CM means ``counting multiplicities'', while IM means ``ignoring multiplicities''. If $g-a=0$ whenever $f-a=0$, then we write $f=a\Rightarrow g=a$.\par 

\smallskip
Looking at an algebraic differential equation 
\bea\label{es1}P\big(z,f,f^{(1)},\ldots, f^{(k)}\big)=0\eea
where $P$ is a polynomial in the variables $f,f^{(1)},\ldots, f^{(k)}$ with meromorphic coefficients, it is not easy to decide, whether Eq. (\ref{es1}) possesses meromorphic solutions. We write (\ref{es1}) in the form 
\[\sideset{}{_{\lambda\in I}}{\sum} a_{\lambda} (f(z))^{i_0}(f^{(1)}(z))^{i_1}\ldots (f^{(k)}(z))^{i_k}=0,\]
where $I$ is a finite set of multi-indices $(i_0,\ldots,i_k)=\lambda$. The degree $|\lambda|$ of a single term in (\ref{es1}) is defined by $|\lambda|: i_0+i_1+\ldots+i_k$ and its weight by $||\lambda||:=i_0+2i_1+\ldots+(k+1)i_k$. 
We start this paper by recalling the classical result of Rellich \cite{FR1}: 
\begin{theoA}\cite{FR1} Let $f$ be an entire solution of $P\big(z,f,f^{(1)},\ldots, f^{(k)}\big)=F(f)$, where the coefficients of $P$ are rational functions and $F$ is transcendental entire. Then $f$ is a constant. 
\end{theoA}

The leading idea, in the research of algebraic differential equations as the pioneering articles by K. Yosida and H. Wittich has been to obtain growth estimates for solutions in terms of the growth of coefficients. Unfortunately, completely general results of this type remain inaccessible. The first important result of a growth estimate type was due to Gol'dberg. In 1956, Gol'dberg \cite{AAG} proved the following result.
\begin{theoB}\cite{AAG} Let $f(z)$ be any meromorphic solution of algebraic differential equation (\ref{es1}) with $k=1$, then the growth order $\rho(f)$ of $f$ satisfies $\rho(f)<\infty$.
\end{theoB}

Second order algebraic differential equations are more problematic than first order  equations especially to obtain growth estimates for solutions. Following result due to Steinmetz \cite{NS1}. 
\begin{theoC}\cite{NS1} All meromorphic solutions f of the homogeneous differential equation (\ref{es1}) with $k=2$ may be represented in the form $f(z)=g_1(z)\exp(g_3(z))/g_2(z)$, where $g_1, g_2, g_3$ are entire functions of finite order of growth. 
\end{theoC}

As everyone know, it is one of the important topics to research the growth of meromorphic solution $f$ of Eq. (\ref{es1}) in $\mathbb{C}$. Some papers that investigate the properties of the growth of solutions of (\ref{es1}) with rational coefficients include \cite{BK1}-\cite{WB1}, \cite{AAG, GLY}, \cite{IL1} and \cite{YXZ, WLL1}.

\smallskip
In this paper we consider a special case of algebraic differential equation 
\bea\label{bm.1} P_k\big(z,f,f^{(1)},\ldots, f^{(k)}\big)=f^{(1)}(f-\mathscr{L}_k(f))-\varphi (f-a)(f-b)=0,\eea
where 
\bea\label{1} \mathscr{L}_k(f)=\sideset{}{_{i=0}^k}{\sum} a_i f^{(i)}\eea
and $\varphi$ is an entire function, $a_i\in\mathbb{C}\;(i=0,1,\ldots, k)$ such that $a_k=1$ and  $a, b\in\mathbb{C}$ such that $a\neq b$.

\smallskip
For the very special case $\mathscr{L}_k(f)=f^{(1)}$, Li and Yang \cite{LY1} solved Eq. (\ref{bm.1}) completely and obtained the following result.

\begin{theoD}\cite[Theorem 2.4]{LY1} Let $f$ be a non-constant meromorphic function satisfying Eq. (\ref{bm.1}) with $\mathscr{L}_k(f)=f^{(1)}$ and $\varphi\not\equiv 0$. Then $\rho(f)=1$ and one of the following cases holds:
\begin{enumerate}
\item[(1)] $ab\not= 0,\varphi\equiv -ab/(a-b)^2$ and $f(z)=a+c\exp(bz/(b-a))$ or $f(z)=b+c\exp(az/(a-b))$, where $c\in\mathbb{C}\backslash \{0\}$;
\item[(2)] $ab =0, \varphi\equiv 1/4$ and $f(z)=(a+b)\left(c\exp\left(z/4\right)-1\right)^2$,
where $c\in\mathbb{C}\backslash \{0\}$.
\end{enumerate}
\end{theoD} 

In general, it is difficult to judge whether Eq. (\ref{bm.1}) has a non-constant solution even for $\mathscr{L}_k(f)=f^{(k)}$. 
Therefore Li and Yang \cite{LY1} posed the following conjecture:

\begin{conjA} For any entire function $\varphi$, the entire solutions of Eq. (\ref{bm.1}) with $\mathscr{L}_k(f)=f^{(k)}$ are functions of exponential type.
\end{conjA}

In 2025, Xu et al. \cite{XMD} settled Conjecture A fully by giving the following result.
\begin{theoE}\cite{XMD} Let $f$ be a non-constant meromorphic solution of Eq (\ref{bm.1}) with $\mathscr{L}_k(f)=f^{(k)}$ such that all the zeros of $f-a$ have multiplicity at least $k$, when $k\geq 2$. Then $f$ is a function of exponential type and one of the following cases holds:
\begin{enumerate}  
\item[(1)] $f(z)=a+A\exp(\lambda z)$, where $A$, $b$ and $\lambda$ are non-zero constants such that $(b-a)\lambda^k=b$,
\item[(2)] $f(z)=b+A\exp(\lambda z)$, where $A$, $a$ and $\lambda$ are non-zero constants with $(a-b)\lambda^k=a$,
\item[(3)] $k=1$, $ab=0$ and $f(z)=(a+b) (c\exp (\frac{z}{4})-1)^2$, where $c$ is a non-zero constant, 
\item[(4)] $k=2$, $a\neq 0$ and $f(z)=c_0\exp(z)+c_1\exp(-z)$, where $c_0$ and $c_1$ are non-zero constants such that $4c_0c_1=a^2$.
\end{enumerate}
\end{theoE}

In the paper, we find out the possible solutions of Eq. (\ref{bm.1}) by giving the following result.
\begin{theo}\label{t1} Let $f$ be a non-constant meromorphic function of Eq. (\ref{bm.1}) such that all the zeros of $f-a$ have multiplicity at least $k$. Then $f$ is a function of exponential type and one of the following cases holds:
\begin{enumerate} 
\item[(1)] $f(z)=A\exp(\lambda z)+a$, where $A, \lambda\in\mathbb{C}\backslash \{0\}$ such that $\sideset{}{_{i=0}^k}{\sum}a_i\lambda^{i}=(b-aa_0)/(b-a)$,

\smallskip
\item[(2)] $f(z)=A\exp(\lambda z)+b$, where $A, \lambda\in\mathbb{C}\backslash \{0\}$ such that $\sideset{}{_{i=0}^k}{\sum}a_i\lambda^{i}=(a-a_0b)/(a-b)$,

\smallskip
\item[(3)] $f(z)=\left((d_0/\lambda)\exp \lambda z+d_1\right)^{k+1}+a$,  where $d_1^{k+1}=b-a$, $\varphi=\lambda$ such that $\lambda^k=(-1)^{k+1}/(k+1)(k+1)!$, $\sideset{}{_{i=0}^k}{\sum}a_i((k+1)\lambda)^i=k$ and $(a_1,\ldots, a_{k-1})\neq (0,0,\ldots,0)$ and $k\geq 2$,

\smallskip
\item[(4)] $f(z)=(c_0\exp \lambda z+c_1)^k+a$, where $a,(a_0-1), c, c_0\in\mathbb{C}\backslash \{0\}$, $(a_0,a_1,\ldots, a_{k-1})\neq (0,\ldots,0)$, $c_1^k=b-a$ and $\varphi=c$ such that $c^k=(-1)^ka(1-a_0)(b-aa_0)^k/k!(b-a)^{k+1}$, $\lambda=(b-a)c/(b-aa_0)$ and $k\lambda-k\lambda\sideset{}{_{i=0}^{k}}
{\sum}a_i(k\lambda)^i-c=0$,

\smallskip
\item[(5)] $f(z)=\left(c_1\exp \lambda_1 z+c_2\exp \lambda_2 z \right)^k+a$, where $a,(a_0-1), c_1,c_2\in\mathbb{C}\backslash \{0\}$,
$(a_{k-2}, a_{k-1})\neq (0,0)$, $\varphi=c\in\mathbb{C}$, $k\geq 2$, $m\in\mathbb{N}$ such that $1\leq m\leq k-1$,
\[\lambda_1=\frac{(m-k)(2(b-a)c-2ka(1-a_0)a_{k-1})}{k^2(k+1)(2m-k)a(1-a_0)}\;\text{and}\;
\lambda_2=\frac{m(2(b-a)c-2ka(1-a_0)a_{k-1})}{k^2(k+1)(2m-k)a(1-a_0)}\]
and $\left(a(1-a_0)/k!c_2^k (\lambda_2-\lambda_1)^k\right)^{\lambda_1}=\left(a(1-a_0)/k!c_1^k (\lambda_1-\lambda_2)^k\right)^{\lambda_2}$. Also if $c\neq 0$, then $c$ satisfies one of the following equations:
\begin{enumerate}
\item[(i)] $\left(\frac{4m(m-k)}{(2m-k)^2}-2k^2+2k\right)(b-a)^2c^2-k\left(\frac{8m(m-k)}{(2m-k)^2}+k^3-2k^2+3k+2\right) a(b-a)a_{k-1}c$ $\\+k^2\left(\frac{4km(m-k)}{(2m-k)^2}+k^3-2k^2+3k+2\right)a^2(1-a_0)a_{k-1}^2-k^3(k+1)^2 a^2(1-a_0)^2a_{k-2}=0$,

\smallskip
\item[(ii)] $\left(\frac{4m(m-k)}{(2m-k)^2}+k(k+1)(k^2+k+2)\right)(b-a)^2c^2$\\ $-k\left(\frac{8m(m-k)}{(2m-k)^2}+2k(k+1)\right)a(b-a)a_{k-1}c+\frac{m(m-k)4k^3}{(2m-k)^2}a^2(1-a_0)^2a_{k-1}^2=0,$

\smallskip
\item[(iii)]
$4\left(\frac{m(m-k)}{(2m-k)^2}-k^2-1\right)(b-a)^2c^2-k\left(\frac{8m(m-k)}{(2m-k)^2}+k^3-6k^2+5k-4\right)a(b-a)a_{k-1}c$ $\\+k^2\left(\frac{4km(m-k)}{(2m-k)^2}+k^3-2k^2+3k-2\right)a^2(1-a_0)a_{k-1}^2+2k(k+1)a(b-a)a_{k-1}-k^3(k+1)^2 a^2(1-a_0)a_{k-2}=0;$
\end{enumerate}

\smallskip
\item[(6)] $f(z)=c_0\exp(z)+c_1\exp(-z)$, where $(a_0,a_1,\ldots, a_{k-1})=(0,0,\ldots,0)$, $k=2$ and $a, c_0, c_1\in\mathbb{C}\backslash \{0\}$ such that $c_0c_1=a^2/4$.
\end{enumerate}
\end{theo}

\begin{rem} Following example shows that the condition ``all the zeros of $f-a$ have multiplicity at least $k$'' in \textrm{Theorem \ref{t1}} is sharp.\end{rem}
\begin{exm} Let $\mathscr{L}_k(f)=f^{(4)}-f^{(3)}-f^{(2)}+f^{(1)}+f$, $f(z)=\exp(z)+\exp(-z)$, $a=4$ and $\varphi=0$. Note that all the zeros of $f-a$ have multiplicity at least $1$ and $c_0c_1=1\neq a^2/4$. Clearly $f$ satisfies Eq. (\ref{bm.1}), but $f$ does not satisfy any case of Theorem \ref{t1}.
\end{exm}

Following corollary  is a consequence of Theorem \ref{t1}.
\begin{cor}\label{c1} Let $f$ be a non-constant entire function, $a$ and $b$ be finite values such that $b\not=a$ and let $k$ be a positive integer. If 
\begin{enumerate}
\item[(i)] all the zeros of $f-a$ have multiplicity at least $k$, 
\item[(ii)] $f=a\Rightarrow \mathscr{L}_k(f)=a$ and 
\item[(iii)] $f=b\Rightarrow \mathscr{L}_k(f)=b$. 
\end{enumerate}
Then $f$ is a function of exponential type and the conclusions of Theorem \ref{t1} hold.
\end{cor}

\begin{rem} Corollary \ref{c1} improves and generalises the results obtained in \cite{JLY1, LXY1} and \cite{XMD}. 
\end{rem}

\section {{\bf Auxiliary lemmas}} 
Let $\mathcal{F}$ be a family of meromorphic functions in a domain $D\subset \mathbb{C}$. We say that $\mathcal{F}$ is normal in $D$ if every sequence $\{f_{n}\}_{n}\subseteq \mathcal{F}$ contains a subsequence which converges spherically and uniformly on the compact subsets of $D$ (see \cite{JS1}). That the limit function is either meromorphic in $D$ or identically equal to $\infty$.
The spherical derivative of $f(z)$ is defined by
\beas f^{\#}(z)=\frac{|f^{(1)}(z)|}{1+|f(z)|^{2}}.\eeas

The following result is the well known Marty's Criterion.
\begin{lem}\label{l2.1}\cite{JS1} A family $\mathcal{F}$ of meromorphic functions on a domain $D$ is normal if and only if for each compact subset $K\subseteq D$, $\exists$ $M\in\mathbb{R}^+$ such that $f^{\#}(z)\leq M$ $\forall$ $f\in\mathcal{F}$ and $z\in K$.
\end{lem} 

We recall the well-known Zalcman's lemma \cite{LZ1}.
\begin{lem}\label{l2.2}\cite[Zalcman's lemma]{LZ1} Let $\mathcal{F}$ be a family of meromorphic functions in the unit disc $\Delta$ such that all zeros of functions in $\mathcal{F}$ have multiplicity greater than or equal to $l$ and all poles of functions in $\mathcal{F}$ have multiplicity greater than or equal to $j$ and $\alpha\in\mathbb{R}$ such that $-l<\alpha<j$. Then $\mathcal{F}$ is not normal in any neighborhood of $z_{0}\in \Delta$ if and only if there exist
\begin{enumerate}\item[(i)] points $z_{n}\in \Delta$, $z_{n}\rightarrow z_{0}$,
\item[(ii)] positive numbers $\rho_{n}$, $\rho_{n}\rightarrow 0^{+}$ and 
\item[(iii)] functions $f_{n}\in \mathcal{F}$,\end{enumerate}
such that $\rho_{n}^{\alpha} f_{n}(z_{n}+\rho_{n} \zeta)\rightarrow g(\zeta)$ spherically locally uniformly in $\mathbb{C}$, where $g$ is a non-constant meromorphic function such that $g^{\#}(\zeta)\leq g^{\#}(0)=1 (\zeta \in \mathbb{C})$.
\end{lem}

\smallskip
We may take $-1<\alpha<1$, if there is no restrictions on the zeros and poles of functions in $\mathcal{F}$. If all functions in $\mathcal{F}$ are holomorphic, then we may take $-1<\alpha<\infty$. On the other hand we may choose $-\infty<\alpha<1$ for families of meromorphic functions which do not vanish.

\smallskip
The following result due to Chang and Zalcman \cite{CZ1}.

\begin{lem}\label{l2.3}\cite[Lemma 2]{CZ1} Let $f$ be a non-constant entire function such that $N(r,f)=O(\log r)$ as $r\to\infty$. If $f$ has bounded spherical derivative on $\mathbb{C}$, then $\rho(f)\leq 1$.
\end{lem}

\begin{lem}\label{l2.4} Let $f$ be an entire function, $\mathscr{L}_k(f)$ be defined as in (\ref{1}) and let $a, b\in\mathbb{C}$ such that $a\neq b$. If all the zeros of $f-a$ have multiplicity at least $k$, $f=a\Rightarrow \mathscr{L}_k(f)=a$ and $f=b\Rightarrow  \mathscr{L}_k(f)=b$, then $\rho(f)\leq 1$. 
\end{lem}

\begin{proof} Set $\mathcal{F}=\{f_{\omega}\}$, where $f_{\omega}(z)=f(\omega+z)$, $z\in\mathbb{C}$. Clearly $\mathcal{F}$ is a family of entire functions defined on $\mathbb{C}$. By assumption, $f(\omega+z)=a\Rightarrow \mathscr{L}_k(f(\omega+z))=a$, $f(\omega+z)=b\Rightarrow \mathscr{L}_k(f(\omega+z))=b$ and all the zeros of $f(\omega+z)-a$ have multiplicity at least $k$.
Since normality is a local property, it is enough to show that $\mathcal{F}$ is normal at each point $z_0\in \mathbb{C}$.
Suppose on the contrary that $\mathcal{F}$ is not normal at $z_0$. Again since normality is a local property, we may assume that $\mathcal{F}$ is a family of holomorphic functions in a domain $\Delta=\{z: |z-z_0|<1\}$. Without loss of generality,
we assume that $z_0\in\Delta$. Then by Lemma \ref{l2.2}, there exist a sequence of functions $f_n\in\mathcal{F}$, where $f_n(z)=f(\omega_n+z)$, a sequence of complex numbers, $z_n$, $z_n\rightarrow z_0$ and a sequence of positive numbers $\rho_n$, $\rho_n\rightarrow 0$ such that
\bea\label{tt1.1} F_n(\zeta)=f_n(z_n+\rho_n \zeta)\rightarrow F(\zeta)\eea
locally uniformly in $\mathbb{C}$, where $F$ is a non-constant entire function. Since $F^{\#}(\zeta)\leq 1$, $\forall$ $\zeta\in\mathbb{C}$, by Lemma \ref{l2.3}, we get $\rho(F)\leq 1$. Since $f_n(z_n+\rho_n \zeta)-a\rightarrow F(\zeta)-a$, Hurwitz's theorem guarantees that all the zeros of $F-a$ have multiplicity at least $k$. Again from (\ref{tt1.1}), we obtain
\bea\label{et1.1} F_n^{(i)}(\zeta)=\rho_n^{i}f_n^{(i)}(z_n+\rho_n \zeta)\rightarrow F^{(i)}(\zeta)\eea
locally uniformly in $\mathbb{C}$ for all $i\in\mathbb{N}$.
Now using (\ref{tt1.1}) and (\ref{et1.1}), we have
\bea\label{et1.5} \mathscr{L}_k(F_n(\xi))=\sideset{}{_{i=0}^k}{\sum} a_i\rho_n^i f_n^{(i)}(z_n+\rho_n \xi)\rightarrow \mathscr{L}_k(F(\xi))\eea
locally uniformly in $\mathbb{C}$.

\smallskip
Now we want to prove that $F=b\Rightarrow  F^{(k)}=0$. If $b$ is a Picard exceptional value of $F$, then obviously $F=b\Rightarrow  F^{(k)}=0$. Next we suppose that $b$ is not a Picard exceptional value of $F$.
Let $F(\xi_0)=b$. Hurwitz's theorem implies the existence of a sequence $\xi_n\rightarrow \xi_0$ with $F_n(\xi_n)=f_n(z_n+\rho_n \xi_n)=b$. Since $f=b\Rightarrow \mathscr{L}_k(f)=b$, we get $\mathscr{L}_k(f_n(z_n+\rho_n \xi_n))=b$. Since $\mathscr{L}_k(f_n(z_n+\rho_n \xi_n))=b$, we have $\sum_{i=0}^k a_i \rho_n^k f_n^{(i)}(z_n+\rho_n \xi_n)=\rho_n^k b$ and so
\bea\label{et1.6a} \mathscr{L}_k(F_n(\xi_n))=\sideset{}{_{i=0}^{k-1}}{\sum}a_i (1-\rho_n^{k-i})\rho_n^i f_n^{(i)}(z_n+\rho_n \xi_n)-\rho_n^k b.\eea

Consequently from (\ref{et1.1}) and (\ref{et1.6a}), we get
\bea\label{et1.7a} \mathscr{L}_k(F(\xi_0))=\lim\limits_{n\to \infty} \mathscr{L}_k(F_n(\xi_n))=\sideset{}{_{i=0}^{k-1}}{\sum}a_i F^{(i)}(\xi_0).\eea

Again from (\ref{et1.5}), we get
\bea\label{et1.8a} \mathscr{L}_k(F(\xi_0))=\lim\limits_{n\to \infty} \mathscr{L}_k(F_n(\xi_n))=\sideset{}{_{i=0}^{k}}{\sum}a_i F^{(i)}(\xi_0).\eea

Now from (\ref{et1.7a}) and (\ref{et1.8a}), we obtain $F^{(k)}(\xi_0)=0$. This shows that $F=b\Rightarrow  F^{(k)}=0$.

\medskip
First suppose $a$ is a Picard exceptional value of $F$. Then by Hadamard's Factorization theorem, we have $F(\zeta)-a=A\exp(\lambda \zeta)$, where $A,\lambda\in\mathbb{C}\backslash \{0\}$. In this case $b$ is not a Picard exceptional value of $F$, otherwise by the second fundamental theorem, we get a contradiction. Note that $F(\zeta)-b=A\exp(\lambda \zeta)+a-b$. Since $a\neq b$ and $F=b\Rightarrow F^{(k)}=0$, we get a contradiction.

\medskip
Next suppose $a$ is not a Picard exceptional value of $F$. Let $F(\zeta_0)=a$. Now we consider the following family
\beas \mathcal{G}=\left\lbrace G_{n}(\zeta): G_{n}(\zeta)=\frac{F_n(\zeta)-a}{\sqrt[2k]{\rho_n}}=\frac{f_n(z_n+\rho_n \zeta)-a}{\sqrt[2k]{\rho_n}}\right\rbrace.\eeas

We claim that the $\mathcal{G}$ is not normal at $\zeta_0$. If not, suppose $\mathcal{G}$ is normal at $\zeta_0$. Then for a given sequence of functions $\{G_n\}\subseteq \mathcal{G}$, there exist a subsequence of $\{G_n\}$ say itself such that
\bea\label{et1.3} G_n(\zeta)=\rho_n^{-\frac{1}{2k}}\{F_n(\zeta)-a\}\to G(\xi)\eea
or possibly
\bea\label{et1.4} G_n(\zeta)\to \infty\eea
spherical uniformly on $\mathbb{C}$ as $n\to\infty$.

Note that $F(\zeta_0)=a$ and since $F-a$ is non-constant, we have $F\not\equiv a$. Now from (\ref{tt1.1}) and Hurwitz's theorem, there exist $\zeta_n$, $\zeta_n\to \zeta_0$ and $F_n(\zeta_n)=a$. Consequently
\bea\label{et1.5a} G(\zeta_0)=\lim_{n\to\infty}\rho_n^{-\frac{1}{2k}}\{F_{n}(\zeta_n)-a\}=0.\eea

Now from (\ref{et1.5a}), we see that (\ref{et1.4}) does not hold. Also we know that zeros of a non-constant analytic function are isolated. Therefore we can find that there exists some deleted neighborhood $\Delta_{\delta(\zeta_0)}=\{\zeta: 0<|\zeta-\zeta_0|<\delta(\zeta_0)\}$ of $\zeta_0$ such that $F(\zeta)\neq a$ for all $\zeta\in \Delta_{\delta(\zeta_0)}$, where
$\delta(\zeta_0)\in\mathbb{R}^+$ depends only on $\zeta_0$. Therefore for $\zeta\in \Delta_{\delta(\zeta_0)}$, there exists some positive number $\rho(\zeta)$ depending only on $\zeta$ such that $|F_n(\zeta)-a|\geq \rho(\zeta)$ for sufficiently large values of $n$. Consequently
\[G(\zeta)=\lim_{n\to\infty}\rho_n^{-\frac{1}{2k}}\{F_n(\zeta)-a\}=\infty\]
and so $G(\zeta)=\infty$, which contradicts the facts that $G(\zeta)\not\equiv \infty$. Hence $\mathcal{G}$ is not normal at $\zeta_0$.
Now by Lemma \ref{l2.2}, there exist a sequence of functions $G_n\in\mathcal{G}$, a sequence of complex numbers, $\zeta_n$, $\zeta_n\rightarrow \zeta_0$ and a sequence of positive numbers $\eta_n$, $\eta_n\rightarrow 0$ such that
\bea\label{et1.6} \tilde{G}_n(\xi)=\frac{G_{n}(\zeta_n+\eta_n\xi)}{\sqrt[2k]{\eta_n}}=\frac{f_n(z_n+\rho_n(\zeta_n+\eta_n\xi))-a}{\sqrt[2k]{\rho_n\eta_n}}\rightarrow \tilde G(\xi)\eea
locally uniformly in $\mathbb{C}$, where $\tilde G$ is a non-constant entire function such that 
\bea\label{t1.7}\tilde{G}^{\#}(\xi)\leq \tilde{G}^{\#}(0)=\frac{|\tilde{G}^{(1)}(0)|}{1+|\tilde{G}(0)|^2}=1,\;\; 
\forall\; \xi\in\mathbb{C}.\eea

Now using Lemma \ref{l2.3}, we get $\rho(\tilde{G})\leq 1$. On the other hand, Hurwitz's theorem guarantees that all the zeros of $\tilde{G}$ have multiplicity at least $k$. Also from the proof of Lemma \ref{l2.2}, we get 
\bea\label{t1.8} \eta_{n}=\frac{1}{G_{n}^{\#}(\zeta_{n})}=\frac{1+|G_n(\zeta_n)|^2}{|G_n^{(1)}(\zeta_n)|}.\eea

It is easy to prove from (\ref{et1.6}) that
\bea\label{t1.8a}(\rho_n\eta_n)^{i-\frac{1}{2k}}f^{(i)}_n(z_n+\rho_n(\zeta_n+\eta_n\xi))= \eta_n^{i-\frac{1}{2k}}G_n^{(i)}(\zeta_n+\eta_n\xi)\to \tilde{G}^{(i)}(\xi)\eea
locally uniformly in $\mathbb{C}$, where $i\in\mathbb{N}$.
We now prove that $\tilde{G}=0\Rightarrow  \tilde{G}^{(k)}=0$. If $0$ is a Picard exceptional value of $\tilde{G}$, then $\tilde{G}=0\Rightarrow  \tilde{G}^{(k)}=0$ is true. Suppose $0$ is not a Picard exceptional value of $\tilde{G}$.
Let $\tilde{G}(\xi_0)=0$. Hurwitz's theorem implies the existence of a sequence $\xi_n\rightarrow \xi_0$ such that $\tilde{G}_n(\xi_n)=0$
and so from (\ref{et1.6}), we have 
\bea\label{ms} f_n(z_n+\rho_n(\zeta_n+\eta_n\xi_n))=a.\eea

By the given conditions we have $f=a\Rightarrow \mathscr{L}_k(f)=a$ and all the zeros of $f-a$ have multiplicities at least $k$. Therefore by a simple calculation, we can prove that $f=a\Rightarrow f^{(k)}=a_1$, where $a_1=a(1-a_0)$ and so from (\ref{ms}), we have  $f^{(k)}_n(z_n+\rho_n(\zeta_n+\eta_n\xi_n))=a_1$. Finally from (\ref{t1.8a}), we have
\[\tilde{G}^{(k)}(\xi_0)=\lim\limits_{n\rightarrow \infty} (\rho_n\eta_n)^{k-\frac{1}{2k}}f_n^{(k)}(z_n+\rho_n(\zeta_n+\eta_n\xi_n))=\lim\limits_{n\rightarrow \infty} (\rho_n\eta_n)^{k-\frac{1}{2k}} a_1=0.\]

Hence $\tilde{G}=0\Rightarrow  \tilde{G}^{(k)}=0$. Since all the zeros of $\tilde{G}$ have multiplicity at least $k$ and $\tilde{G}=0\Rightarrow  \tilde{G}^{(k)}=0$, we can conclude that all the zeros of $\tilde{G}$ have multiplicity at least $k+1$. 

Since $\tilde{G}^{\#}(0)=1$, from (\ref{t1.7}), we get $\tilde{G}^{(1)}(0)\neq 0$.  Note that all the zeros of $\tilde{G}$ have multiplicity at least $k+1$. If $\tilde{G}(0)=0$, then obviously $\tilde{G}^{(1)}(0)=0$, which is impossible. Hence $\tilde{G}(0)\neq 0$. 
Now from (\ref{et1.6}) and (\ref{t1.8a}), we have 
\bea\label{et1.8b} \frac{\tilde{G}_{n}^{(1)}(\xi)}{\tilde{G}_{n}(\xi)}=\eta_{n}\frac{G_{n}^{(1)}(\zeta_{n}+\eta_{n}\xi)}{G_{n}(\zeta_{n}+\eta_{n}\xi)}\rightarrow \frac{\tilde{G}^{(1)}(\xi)}{\tilde{G}(\xi)},\eea
locally uniformly in $\mathbb{C}$. Clearly from (\ref{t1.8}) and (\ref{et1.8b}), we get 
\beas \eta_{n}\left|\frac{G_{n}^{(1)}(\zeta_{n})}{G_{n}(\zeta_{n})}\right|=\frac{1+|G_{n}(\zeta_{n})|^{2}}{|G_{n}^{(1)}(\zeta_{n})|}\frac{|G_{n}^{(1)}(\zeta_{n})|}{|G_{n}(\zeta_{n})|}
=\frac{1+|G_{n}(\zeta_{n})|^{2}}{|G_{n}(\zeta_{n})|}\rightarrow \left|\frac{\tilde{G}^{(1)}(0)}{\tilde{G}(0)}\right|\neq 0,\infty,\eeas
which implies that $\lim\limits_{n\rightarrow \infty} G_{n}(\zeta_{n})\not=0,\infty$ and so from (\ref{et1.6}), we get $\tilde{G}_{n}(0)=\eta_{n}^{-\frac{1}{2k}}G_{n}(\zeta_{n}) \rightarrow \infty.$

Again from (\ref{et1.6}), we get $\tilde{G}_{n}(0)\rightarrow \tilde{G}(0)\neq 0,\infty$. Therefore we get a contradiction.\par

\smallskip
Therefore all the foregoing discussion shows that $\mathcal{F}$ is normal at $z_0$. Consequently $\mathcal{F}$ is normal in $\mathbb{C}$. Hence by Lemma \ref{l2.1}, there exists $M>0$ satisfying $f^{\#}(\omega)=f^{\#}_{\omega}(0)<M$ for all $\omega\in\mathbb{C}$. Consequently by Lemma \ref{l2.3}, we get $\rho(f)\leq 1$.
This completes the proof.
\end{proof}

\begin{lem}\label{l2.5} \cite[Theorem 1.2]{LY} Let $f$ be a non-constant entire function, $k$ be a positive integer and let $Q$ be a polynomial such that all the zeros of $f-Q$ have multiplicity at least $k$. If $f\equiv f^{(k)}$, then one of the following cases must occur:
\begin{enumerate}
\item[(1)] $k=1$ and $f(z)=A\exp z$, where $A\in\mathbb{C}\backslash \{0\}$,
\item[(2)] $k=2$, $Q$ reduces to a constant. If $Q\equiv 0$, then $f(z)=A\exp \lambda z$, where $A, \lambda\in\mathbb{C}\backslash \{0\}$ such that $\lambda^2=1$. If $Q\equiv a\in\mathbb{C}\backslash \{0\}$, then $f(z)=c_0\exp z+c_1\exp -z$, where $c_0, c_1\in\mathbb{C}\backslash \{0\}$ such that $c_0c_1=a^2/4$,
\item[(3)] $k\geq 3$, $Q$ reduces to $0$ and $f(z)=A\exp \lambda z$, where $A, \lambda\in\mathbb{C}\backslash \{0\}$ such that $\lambda^k=1$.
\end{enumerate}
\end{lem}

\begin{lem}\label{l1} \cite[Theorem 4.1]{HKR,NO} Let $f$ be a non-constant entire function such that $\rho(f)\leq 1$ and $k\in\mathbb{N}$. Then $m\big(r,\frac{f^{(k)}}{f}\big)=o(\log r)$ as $r\to \infty$.
\end{lem}

\begin{lem}\label{l2}\cite[Lemma 2]{CCY} Let $f$ be a non-constant meromorphic function and let $a_{n}(\not\equiv 0), 
a_{n-1},\ldots,a_{0}$ be small functions of $f$.  Then $T(r,\sum_{i=0}^{n}a_{i}f^{i})= nT(r,f) + S(r,f).$ \end{lem}

\section {{\bf Proof of the main result}} 
\begin{proof}[{\bf Proof of Theorem \ref{t1}}] Let $f$ be a non-constant meromorphic solution of the equation (\ref{bm.1}).
Now we divide following two cases.\par

\smallskip
{\bf Case 1.} Let $\varphi\not\equiv 0$. Since $\varphi$ is an entire function, (\ref{bm.1}) shows that $f$ is also a non-constant entire function. Now we prove that $f=a\Rightarrow  \mathscr{L}_k(f)=a$. If $a$ is a Picard exceptional value of $f$, then $f=a\Rightarrow  \mathscr{L}_k(f)=a$ is true. Suppose $a$ is not a Picard exceptional value of $f$. Let $z_0$ be a zero of $f-a$ of multiplicity $p_0(\geq k)$. Clearly $z_0$ is a zero of $f^{(i)}$ of multiplicity $p_0-i$, where $1\leq i\leq k$. Then (\ref{bm.1}) shows that $z_0$ is a zero of $f-\mathscr{L}_k(f)$ and so $z_0$ is also a zero of $\mathscr{L}_k(f)-a$. Therefore $f=a\Rightarrow  \mathscr{L}_k(f)=a$. Since all the zeros of $f-a$ have multiplicity at least $k$ and $\mathscr{L}_k(f)=\sum_{j=0}^k a_j f^{(j)}$, we have 
\bea\label{rs} f=a\Rightarrow f^{(k)}=a(1-a_0).\eea

Similarly $f=b\Rightarrow  \mathscr{L}_k(f)=b$. Now by Lemma \ref{l2.4}, we have $\rho(f)\leq 1$.
Note that
\bea\label{ebmm.0}\varphi=\left(af^{(1)}/(f-a)-bf^{(1)}/(f-b)\right)\left(1-\sideset{}{_{i=1}^k}{\sum}a_if^{(i)}/f\right)/(a-b).\eea

Therefore using Lemma \ref{l1} to (\ref{ebmm.0}), we get $m(r,\phi)=o(\log r)$ and so $T(r,\varphi)=o(\log r)$, which implies that $\varphi$ is a constant, say $c\in\mathbb{C}\backslash \{0\}$. Then from (\ref{bm.1}), we have 
\bea\label{bm.1a} f^{(1)}\left(f-\mathscr{L}_k(f)\right)=c(f-a)(f-b),\eea
which shows that $f$ is a transcendental entire function.
Differentiating (\ref{bm.1a}) once, we get
\bea\label{bm.1aa} f^{(2)}(f-\mathscr{L}_k(f))+f^{(1)}(f^{(1)}-\mathscr{L}_k^{(1)}(f))=c f^{(1)}(f-a+f-b).\eea

Denote by $S_{(m,n)}(a, a(1-a_0))$ the set of those points $z\in\mathbb{C}$ such that $z$ is an $a$-point of $f$ of order $m$ and an $a(1-a_0)$-point of $f^{(k)}$ of order $n$. 

\smallskip   
Now we consider following two sub-cases.\par

\smallskip
{\bf Sub-case 1.1.} Let $a(1-a_0)=0$. Then from (\ref{rs}), we have $f=a\Rightarrow f^{(k)}=0$ and so all the zeros of $f-a$ have multiplicity at least $k+1$. Let $z_{m,n}\in S_{(m,n)}(a,0)\;(m\geq k+1)$. Then $f(z_{m,n})=0$ and $f^{(k)}(z_{m,n})=0$. 
Observe that if $m\geq k+2$, then from (\ref{bm.1a}), we get a contradiction. Hence $m=k+1$ and so all the zeros of $f-a$ have multiplicity exactly $k+1$. Let $z_{k+1,n}\in S_{(k+1,n)}(a,0)$. If $n\geq 2$, then from (\ref{bm.1a}), we get a contradiction. So $n=1$. 
Locally
\bea\label{rs1}f(z)-a=A(z-z_{k+1,1})^{k+1}+O\left((z-z_{k+1,1})^{k+2}\right).\eea

Now using (\ref{rs1}) to (\ref{bm.1aa}), one can easily get  
$f^{(k+1)}(z_{k+1,1})=(b-a)c/(k+1)$. Consequently 
\bea\label{ebmm12.2} f=a\Rightarrow f^{(k+1)}=(b-a)c/(k+1).\eea

\medskip
First suppose $f\neq a$. Then $f(z)=a+A\exp(\gamma z)$ and so $f^{(i)}(z)=A\gamma^i\exp \gamma z$ for $i=1,2,\ldots,k$, where $A, \gamma\in\mathbb{C}\backslash \{0\}$. Clearly $\mathscr{L}_k(f(z))=A\sum_{i=0}^k a_i \gamma^i \exp(\gamma z)+a$ and $b$ is not a Picard exceptional values of $f$. Let $f(z_1)=b$. Since $f=b\Rightarrow \mathscr{L}_k(f)=b$, we have  $f(z_1)=a+A\exp(\gamma z_1)=b$ and $\mathscr{L}_k(f(z_1))=A\sum_{i=0}^k a_i \gamma^i\exp(\gamma z_1)+a=b$. Now eliminating $\exp(\gamma z_1)$, we get $\sum_{i=0}^k a_i\gamma^i=1$. Therefore $f(z)=a+A\exp(\gamma z)$, where $A,\gamma\in\mathbb{C}\backslash \{0\}$ such that $\sum_{i=1}^k a_i\gamma^{i-1}=0$.

\medskip
Next suppose $f\neq b$. Clearly $a$ is not a Picard exceptional value of $f$. In this case also, we obtain $f(z)=b+A\exp(\gamma z)$, where $A, \gamma\in\mathbb{C}\backslash \{0\}$ such that $\sum_{i=1}^k a_i \gamma^{i-1}=0$.

\medskip
Henceforth we suppose $a$ and $b$ are not Picard exceptional values of $f$. Since all the zeros of $f-a$ have multiplicity exactly $k+1$, we take $f-a=g^{k+1}$, where $g$ is a transcendental entire function having only simple zeros. Now from the proof of the Theorem 1.1 \cite{MSS1}, we see that
\bea\label{ebbm.2} f^{(k)}=(k+1)!g(g^{(1)})^{k} +k(k-1)(k+1)!g^{2}(g^{(1)})^{k-2}g^{(2)}/4+R_1(g),\eea
\bea\label{ebbm.3} f^{(k + 1)}=(k + 1)!(g^{(1)})^{k + 1} +k(k + 1)(k + 1)!g(g^{(1)})^{k - 1}g^{(2)}/2+R_2(g),\eea  
\be\label{ebbm.3a} f^{(k + 2)}=(k + 1)!(k+1)(k+2)(g^{(1)})^{k}g^{(2)}/2+R_3(g),\ee  
where $R_{i}(g)$'s are differential polynomial in $g$, where $i=1,2,3$ and each term of $R_{1}(g)$ contains $g^{m} (3 \leq m \leq k$) as a factor. On the other hand from (\ref{bm.1}), we have
\bea\label{ebbm.4} (k+1)g^{k+1}g^{(1)}-(k+1)g^{(1)}\sideset{}{_{i=0}^{k}}{\sum}a_i(g^{k+1})^{(i)}-cg^{k+2}=c(a-b)g.\eea

Denote by $N(r,0;g^{(1)}\mid g^{k+1}\neq b-a)$ the counting function of those zeros of $g^{(1)}$ which are not the zeros of $g^{k+1}+a-b$.
 We denote by $N(r,b-a;g^{k+1}\mid\geq 2)=0$ the counting function of multiple $b-a$-points of $g^{k+1}$.
 
\smallskip
Now we divide following two sub-cases.\par

\smallskip
{\bf Sub-case 1.1.1.} Let $N(r,b-a;g^{k+1}\mid\geq 2)=0$. Now from (\ref{ebbm.4}), we deduce $N(r,0;g^{(1)}\mid g^{k+1}\neq b-a)=0$.
Since $f-a=g^{k+1}$ and $N(r,b-a;g^{k+1}\mid\geq 2)=0$, it follows that $g^{(1)}\neq 0$. Note that $g$ is a transcendental entire function such that $\rho(g)\leq 1$ and $g^{(1)}\neq 0$. So we can take 
\bea\label{ebbm.5} g^{(1)}(z)=d_0\exp(\lambda z),\eea
where $d_0, \lambda\in\mathbb{C}\backslash \{0\}$. On integration, we have 
\bea\label{ebbm.6} g(z)=(d_0\exp(\lambda z)+\lambda d_1)/\lambda,\eea
where $d_1\in\mathbb{C}$. As $0$ is not a Picard exceptional value of $f-a$, we have $d_1\neq 0$.
Let $g(z_{2})=0$. Clearly $f(z_2)=a$. Now from (\ref{ebmm12.2}), we get $f^{(k+1)}(z_{2})=(b-a)c/(k+1)$ and so
from (\ref{ebbm.3}), get
\bea\label{ebbm.00} (g^{(1)}(z_{2}))^{k+1}=(b-a)c/(k+1)(k+1)!.\eea

Clearly from (\ref{ebbm.5})-(\ref{ebbm.00}), we have $d_0^{k+1}\exp((k+1)\lambda z_{2})=(-\lambda d_1)^{k+1}$ and $d_0^{k+1}\exp ((k+1)\lambda z_{2})=(b-a)c/(k+1)(k+1)!$ and so
\bea\label{ebbm.9} (-\lambda d_1)^{k+1}=(b-a)c/(k+1)(k+1)!.\eea

Again from (\ref{ebbm.6}), we have
\bea\label{rs2} g^{k+1}(z)&=&\left(d_0/\lambda\right)^{k+1}\exp (k+1)\lambda z+\;{}^{k+1}C_{1}\left(d_0/\lambda\right)^{k}d_1\exp k\lambda z+\ldots\\&&+\;{}^{k+1}C_{k-1}\left(d_0/\lambda\right)^{2}d_1^{k-1}\exp 2\lambda z+\;{}^{k+1}C_{k}(d_0d_1^{k}/\lambda) \exp \lambda z+d_1^{k+1}\nonumber\eea
and so
\bea\label{rs3} (g^{k+1})^{(i)}(z)&=&\left(d_0/\lambda\right)^{k+1}((k+1)\lambda)^i\exp (k+1)\lambda z+\;{}^{k+1}C_{1}\left(d_0/\lambda\right)^{k}d_1(k\lambda)^i\exp k\lambda z+\ldots\nonumber\\&&+\;{}^{k+1}C_{k-1}\left(d_0\lambda\right)^{2}d_1^{k-1}(2\lambda)^i\exp 2\lambda z+\;{}^{k+1}C_{k}(d_0d_1^{k}\lambda^i/\lambda)\exp \lambda z,\eea
where $i\in\mathbb{N}$. Consequently
\bea\label{rs4} &&\mathscr{L}_k(g^{k+1}(z))\\&=&\left(d_0/\lambda\right)^{k+1}\sideset{}{_{i=0}^{k}}{\sum}a_i\left((k+1)\lambda\right)^i\exp (k+1)\lambda z\nonumber\\&&+\;{}^{k+1}C_{1}\left(d_0/\lambda\right)^{k}\sideset{}{_{i=0}^{k}}{\sum}a_i\left(k\lambda\right)^i d_1\exp k\lambda z+\ldots\nonumber\\&&+\;{}^{k+1}C_{k-1}\left(d_0/\lambda\right)^{2}\sideset{}{_{i=0}^{k}}{\sum}a_i(2\lambda)^{i} d_1^{k-1}\exp 2\lambda z+(k+1)d_0d_1^{k}/\lambda\sideset{}{_{i=0}^{k}}{\sum}a_i\lambda^{i}\exp \lambda z.\nonumber\eea

Now using (\ref{rs2})-(\ref{rs4}) to (\ref{ebbm.4}), we get
\bea\label{ebbm.10} &&\left\{(k+1)/\lambda^{k+1}-((k+1)/\lambda^{k+1})\sideset{}{_{i=0}^k}{\sum}a_i((k+1)\lambda)^i-c/\lambda^{k+2}\right\}d_{0}^{k+2}\exp (k+2)\lambda z\nonumber\\
&&+\left\{(k+1)^{2}/\lambda^{k}-((k+1)^{2}/\lambda^k)\sideset{}{_{i=0}^{k}}{\sum}a_i(k\lambda)^i-c(k+2)/\lambda^{k+1}\right\}d_{0}^{k+1}d_{1}\exp (k+1)\lambda z\nonumber\\
&&+\ldots+\left\{(k+1)^{2}/\lambda-((k+1)^{2}/\lambda)\sideset{}{_{i=0}^{k}}{\sum}a_i\lambda^i-\;{}^{k+2}C_{k}c/\lambda^{2}\right\}d_0^2d_1^k\exp 2\lambda z\\
&&+\left\{k+1-c(k+2)/\lambda\right\}d_{0}d_{1}^{k+1}\exp \lambda z-c d_{1}^{k+2}=(a-b)cd_0/\lambda\exp \lambda z+(a-b)cd_1,\nonumber\eea
which shows by Lemma \ref{l2} that 
\[(k+1)/\lambda^{k+1}-((k+1)/\lambda^{k+1})\sideset{}{_{i=0}^k}{\sum}a_i((k+1)\lambda)^i-c/\lambda^{k+2}=0,\]
\bea\label{ebbm.11} d_1^{k+1}=b-a\;\;\text{and}\;\;\left\{(k+1)-c(k+2)/\lambda\right\}d_{0}d_{1}^{k+1}=(a-b)cd_0/\lambda,\;\;
\text{i.e.},\;\; c=\lambda.\eea

Now from (\ref{ebbm.9}) and (\ref{ebbm.11}), we have $\lambda^k=(-1)^{k+1}/(k+1)(k+1)!.$
If $(a_0,a_1,\ldots, a_{k-1})=(0,0,\ldots,0)$ and $k\geq 2$, then
\beas &&(k+1)/\lambda^{k+1}-((k+1)/\lambda^{k+1})\sideset{}{_{i=0}^k}{\sum}a_i((k+1)\lambda)^i-c/\lambda^{k+2}\\&=&(k+1)\left((-1)^{k+1}k(k+1)!-(k+1)^k\right)/\lambda\neq 0\nonumber\eeas
for $k\geq 2$ and so we get a contradiction. Finally
\[f(z)=\left((d_0/\lambda)\exp \lambda z+d_1\right)^{k+1}+a,\]
where $d_1^{k+1}=b-a$, $\varphi=\lambda$ such that $\lambda^k=(-1)^{k+1}/(k+1)(k+1)!$, $\sideset{}{_{i=0}^k}{\sum}a_i((k+1)\lambda)^i=k$ and $(a_1,\ldots, a_{k-1})\neq (0,0,\ldots,0)$ and $k\geq 2$.\par

\smallskip
{\bf Sub-case 1.1.2.} Let $N(r,b-a;g^{k+1}\mid\geq 2)\neq 0$. Differentiating (\ref{ebbm.4}) twice, we have respectively
\bea\label{ebbm.14a} &&(k+1)^2g^k (g^{(1)})^2+(k+1)g^{k+1}g^{(2)}-(k+1)g^{(1)}\sideset{}{_{i=0}^k}{\sum} a_i (g^{k+1})^{(i+1)}\\&&-(k+1)g^{(2)}\sideset{}{_{i=0}^k}{\sum} a_i (g^{k+1})^{(i)}
-c(k+2)g^{k+1}g^{(1)}=(a-b)cg^{(1)}\nonumber\eea
and
\bea\label{ebbm.14} &&(k+1)^2kg^{k-1}(g^{(1)})^3+3(k+1)^2 g^kg^{(1)}g^{(2)}+(k+1)g^{k+1}g^{(3)}\nonumber\\&&
-(k+1)g^{(3)}\sideset{}{_{i=1}^k}{\sum}a_i(g^{k+1})^{(i)}-2(k+1)g^{(2)}\sideset{}{_{i=1}^k}{\sum}a_i(g^{k+1})^{(i+1)}\nonumber\\&&-(k+1)g^{(1)}\sideset{}{_{i=1}^k}{\sum}a_i(g^{k+1})^{(i+2)}-c(k+2)(k+1)g^k(g^{(1)})^2
-c(k+2)g^{k+1}g^{(2)}\nonumber\\&&=(a-b)cg^{(2)}.
\eea

Now using (\ref{ebbm.2})-(\ref{ebbm.3a}) to (\ref{ebbm.14}), we get
\bea\label{ebbm.15}&&-(k+1)(k+1)!\left((k^2+3k+6)(g^{(1)})^{k+1}g^{(2)}/2+a_{k-1}(g^{(1)})^{k+2}\right)+R_{4}(g)\\&=&(a-b)cg^{(2)}\nonumber,\eea
where $R_{4}(g)$ is a differential polynomial in $g$. Let $z_2$ be a zero of $g$. Then (\ref{ebbm.15}) gives 
\bea\label{ebbm.17} (k^2+3k+6)(g^{(1)}(z_2))^{k+1}g^{(2)}(z_2)+2a_{k-1}(g^{(1)}(z_2))^{k+2}=\frac{2(b-a)c}{(k+1)(k+1)!}g^{(2)}(z_2).\eea
 
Therefore from (\ref{ebbm.00}) and (\ref{ebbm.17}), we get $((k^2+3k+4)/2)g^{(2)}(z_2)+a_{k-1}g^{(1)}(z_2)=0$.
This shows that 
\[g=0\Rightarrow ((k^2+3k+4)/2)g^{(2)}+a_{k-1}g^{(1)}=0.\]

Let 
\bea\label{ebbm.18} H_1=((k^2+3k+4)g^{(2)}+2a_{k-1}g^{(1)})/g.\eea

We now divide following sub-cases.\par

\smallskip
{\bf Sub-case 1.1.2.1.} Let $H_1\equiv 0$. Then (\ref{ebbm.18}), gives 
\[g^{(2)}+\hat a g^{(1)}\equiv 0,\]
where $\hat{a}=2a_{k-1}/(k^2+3k+4)$. Since $g$ is transcendental, it follows that $\hat{a}\neq 0$. On integrating, we get $g(z)=\hat {A}_0\exp(-\hat{a} z)+\hat{B}_0$ where $\hat{A}_0, \hat{B}_0\in\mathbb{C}\backslash \{0\}$. Also we have $f-b=g^{k+1}+a-b$ and so $f^{(1)}=(k+1)g^kg^{(1)}$. Since $g^{(1)}\neq 0$, we have $N(r,b;f\mid \geq 2)=0$, which is impossible.\par

\smallskip
{\bf Sub-case 1.1.2.2.} Let $H_1\not\equiv 0$. Since $g$ has only simple zeros and 
\[g=0\Rightarrow ((k^2+3k+4)/2)g^{(2)}+a_{k-1}g^{(1)}=0,\]
we see that $H_1$ is an entire function and so from (\ref{ebbm.18}), one can conclude that $H_1$ is a constant, say $\delta\in\mathbb{C}\backslash \{0\}$ and so from (\ref{ebbm.18}), we get
\bea\label{ebbm.19} g^{(2)}=\alpha g^{(1)}+\beta g,\;\;
\text{where}\;\; \alpha =-\hat{a}\;\;\text{and}\;\;\beta =\delta/(k^2+3k+4)\neq 0.\eea 

Differentiating (\ref{ebbm.19}) and using it repeatedly, we have
\bea\label{ebbm.20} g^{(i)}=\alpha_{i-1} g^{(1)}+\beta_{i-1} g,\eea
where $i\geq 2$ and $\alpha_{i-1}, \beta_{i-1}\in\mathbb{C}$. Now using (\ref{ebbm.20}) to (\ref{ebbm.2}), we can assume that
\bea\label{ebbm.21} \sideset{}{_{i=0}^k}{\sum}a_i (g^{k+1})^{(i)}=\check{c}_0 g(g^{(1)})^k+\check{c}_1 g^2(g^{(1)})^{k-1}+\ldots+\check {c}_k g^kg^{(1)}+\check{c}_{k+1} g^{k+1},\eea
where $\check{c}_0(=(k+1)!), \ldots,\check{c}_{k+1}\in\mathbb{C}$. Consequently from (\ref{ebbm.4}) and (\ref{ebbm.21}), we get 
\bea\label{ebbm.22} (k+1)\left(\check{c}_0(g^{(1)})^{k+1}+\ldots+\check{c}_k g^{k-1}(g^{(1)})^2+(1-\check{c}_{k+1})g^kg^{(1)}\right)+cg^{k+1}=(b-a)c.\eea

Now differentiating (\ref{ebbm.21}) and using (\ref{ebbm.19}) and (\ref{ebbm.20}), one can easily assume that
\bea\label{ebbm.23} \sideset{}{_{i=0}^k}{\sum}a_i (g^{k+1})^{(i+1)}=\check{d}_0 g(g^{(1)})^k+\check{d}_1 g^2(g^{(1)})^{k-1}+\ldots+\check{d}_k g^kg^{(1)}+\check{d}_{k+1} g^{k+1},\eea
where $\check{d}_0, \ldots,\check{d}_{k+1}(=\check{c}_k\beta)\in\mathbb{C}$. So using (\ref{ebbm.19}), (\ref{ebbm.21}) and (\ref{ebbm.23}) to (\ref{ebbm.14a}), we get
\bea\label{ebbm.24}\check{e}_0 (g^{(1)})^{k+2}+\check{e}_1 g(g^{(1)})^{k+1}+\ldots+\check{e}_{k+1} g^{k+1}g^{(1)}+\check{e}_{k+2} g^{k+2}=(b-a)cg^{(1)}, \eea
where $\check{e}_0=(k+1)(k+1)!$, 
\[\check{e}_{k+1}=(k+1)(-\alpha+\check{d}_{k+1}+\alpha\check{c}_{k+1}+c(k+2))=(k+1)(-\alpha+\check{c}_k\beta+\alpha\check{c}_{k+1}+c(k+2))\]
 and $\check{e}_{k+2}=(k+1)\beta(\check{c}_{k+1}-1).$

Let $z_3$ be a multiple zero of $f-b=g^{k+1}+a-b$. Then $g(z_3)\neq 0$ and $g^{(1)}(z_3)=0$. Clearly from (\ref{ebbm.24}), we get $\check{e}_{k+2}=0$. Since $\beta\neq 0$, we have $\check{c}_{k+1}=1$. Consequently from (\ref{ebbm.24}), we get 
\bea\label{ebbm.25}\check{e}_0 (g^{(1)})^{k+1}+\check{e}_1 g(g^{(1)})^{k}+\ldots+\check{e}_{k+1} g^{k+1}=(b-a)c. \eea

Let $z_3$ be a multiple zero of $f-b=g^{k+1}+a-b$. Clearly $g^{k+1}(z_3)=b-a$ and $g^{(1)}(z_3)=0$. Then from (\ref{ebbm.25}), we conclude that $\check{e}_{k+1}=c$ and so
\bea\label{ebbm.26} (k+1)(\check{c}_k\beta+c(k+2))=c.\eea

On the other hand from (\ref{ebbm.22}), we have 
\bea\label{ebbm.27} (k+1)\check{c}_0(g^{(1)})^{k+1}+\ldots+(k+1)\check{c}_k g^{k-1}(g')^2+cg^{k+1}=(b-a)c.\eea

Now differentiating (\ref{ebbm.27}) and using (\ref{ebbm.19}), one can easily assume that
\bea\label{ebbm.28} \check{f}_0(g^{(1)})^{k}+\ldots+(k+1)(2\check{c}_k\beta+c) g^{k}=0.\eea

If $z_3$ is a multiple zero of $f-b$, then from (\ref{ebbm.28}), we get $2\check{c}_k\beta+c=0$ and so from (\ref{ebbm.26}), we get a contradiction.\par

\smallskip
{\bf Sub-case 1.2.} Let $a(1-a_0)\neq 0$. 
Let $z_{m,n}\in S_{(m,n)}(a, a(1-a_0))\;(m\geq k)$. Clearly $f(z_{m,n})=a$ and $f^{(k)}(z_{m,n})=a(1-a_0)$. If $m\geq k+1$, then $f^{(k)}(z_{m,n})=0$ and so we get a contradiction. Hence $m=k$ and so all the zeros of $f-a$ have multiplicity exactly $k$.
Let $z_{k,n}\in S_{(k,n)}(a, a(1-a_0))$. Locally
\bea\label{rs5}f(z)-a=A(z-z_{k,n})^{k}+O\left((z-z_{k,n})^{k+1}\right)\eea
and
\bea\label{rs6} f^{(k)}(z)-a(1-a_0)=\tilde b_n(z-z_{k,n})^{n}+O\left((z-z_{k,n})^{n+1}\right).\eea

First suppose $n=1$. Then $z_{k,1}$ is a simple zero of both $f^{(k)}-a$ and $\mathscr{L}_k(f)-a$. 
Now from (\ref{rs}) and (\ref{rs5}), we get $f^{(k)}(z_{k,1})=a(1-a_0)$ and $f^{(k)}(z_{k,1})=k!A$ respectively. So $k!A=a(1-a_0)$.
Therefore using (\ref{rs5}) and (\ref{rs6}) to (\ref{bm.1aa}) and then simplifying, we get 
\beas \left(k(\tilde {b}_1+a(1-a_0)a_{k-1})+(a-b)c\right)(z-z_{k,1})^{k-1}+O\left((z-z_{k,1})^k\right)\equiv 0,\eeas
which gives $\tilde {b}_1=-((a-b)c+ka(1-a_0)a_{k-1})/k$, i.e., 
$f^{(k+1)}(z_{k,1})=-((a-b)c+ka(1-a_0)a_{k-1})/k.$ 

Next suppose $n\geq 2$. Then from (\ref{bm.1a}), we get $(a-b)c+ka(1-a_0)a_{k-1}=0$. 

Consequently 
\bea\label{ebm1.2} f=a\Rightarrow f^{(k+1)}=-((a-b)c+ka(1-a_0)a_{k-1})/k.\eea

\smallskip
Note that if $a$ is a Picard exceptional value of $f$, then as usual we obtain
$f(z)=a+A\exp(\delta z)$, where $A,\delta\in\mathbb{C}\backslash \{0\}$ such that 
\[\sideset{}{_{j=0}^k}{\sum}a_j\delta^j=(b-aa_0)/(b-a).\]

Also if $b$ is a Picard exceptional value of $f$, then $f(z)=b+A\exp(\delta z)$, where $A, \delta\in\mathbb{C}\backslash \{0\}$ such that 
\[\sideset{}{_{j=0}^k}{\sum} a_j\delta^j=(a-a_0b)/(a-b).\]

Henceforth we assume that $a$ and $b$ are not Picard exceptional values of $f$. 

\smallskip
Since all the zeros of $f-a$ have multiplicity exactly $k$, we take $f-a=g^k$, where $g$ is a transcendental entire function having only simple zeros. Now from the proof of the Theorem 1.1 \cite{MSS1}, we see that 
\bea\label{ebm.2} f^{(k)}=(g^k)^{(k)}=k!(g^{(1)})^{k}+\frac{k(k-1)}{2}k!g(g^{(1)})^{k-2}g^{(2)}+ R_{5}(g),\eea 
\bea\label{ebm.2a} f^{(k-1)}=(g^k)^{(k-1)}=k!g(g^{(1)})^{k-1}+R_6(f)\eea
and 
\be\label{ebm.3} f^{(k+1)}=(g^k)^{(k+1)} = \frac{k(k+1)}{2}k!(g^{(1)})^{k-1}g^{(2)}+ R_{7}(g),\ee 
where $R_{i}(g)$'s are differential polynomials in $g$, where $i=5,6,7$ such that each term of $R_{5}(g)$ contains $g^{m} (1 \leq m \leq k-1$) as a factor. Also from (\ref{bm.1a}), we get
\bea\label{ebm.4} kg^kg^{(1)}-kg^{(1)}\sideset{}{_{i=1}^k}{\sum} a_i (g^k)^{(i)}-cg^{k+1}=c(a-b)g-ka(1-a_0)g^{(1)}.\eea

Differentiating (\ref{ebm.4}) once again, we get
\bea\label{ebm.5} &&k^2g^{k-1}(g^{(1)})^2+kg^kg^{(2)}-kg^{(2)}\sideset{}{_{i=1}^k}{\sum}a_i (g^k)^{(i)}-kg^{(1)}\sideset{}{_{i=1}^k}{\sum} a_i (g^k)^{(i+1)}-c(k+1)g^kg^{(1)}\nonumber\\&&=c(a-b)g^{(1)}-ka(1-a_0)g^{(2)}.\eea

Now using (\ref{ebm.2}) to (\ref{ebm.4}), we get 
\bea\label{ebm.6}-kk!(g^{(1)})^{k+1}+R_{8}(g)=c(a-b)g-ka(1-a_0)g^{(1)},\eea
where $R_{8}(g)$ is a differential polynomial. Again using (\ref{ebm.2}), (\ref{ebm.3}) to (\ref{ebm.5}), we get
\bea\label{ebm.7}-kk!a_{k-1}(g^{(1)})^{k+1}-kk!\frac{k^2+k+2}{2}(g^{(1)})^kg^{(2)}+R_9(g) =c(a-b)g^{(1)}-ka(1-a_0)g^{(2)},\eea
where $R_{9}(g)$ is a differential polynomial.

Let $z_2$ be a zero of $g$. Then from (\ref{ebm.6}) and (\ref{ebm.7}), we have respectively
\bea\label{ebm.8} (g^{(1)}(z_2))^k=a(1-a_0)/k!\eea
and 
\beas\label{ebm.9} a_{k-1}(g^{(1)}(z_2))^{k+1}+\frac{k^2+k+2}{2}(g^{(1)}(z_2))^kg^{(2)}(z_2)=-\frac{c(a-b)}{kk!}g^{(1)}(z_2)+\frac{a(1-a_0)}{k!}g^{(2)}(z_2)\eeas
and so $k^2(k+1)a(1-a_0)g^{(2)}(z_2)+2\left(ka_{k-1} a(1-a_0)+c(a-b)\right)g^{(1)}(z_2)=0.$
This shows that
$g=0\Rightarrow k^2(k+1)a(1-a_0)g^{(2)}+2\left(k a_{k-1} a(1-a_0)+c(a-b)\right)g^{(1)}=0$.
Let 
\bea\label{ebm.10} H_2=\frac{k^2(k+1)a(1-a_0)g^{(2)}+2\left(k a_{k-1} a(1-a_0)+c(a-b)\right)g^{(1)}}{g}.\eea

We now divide following sub-cases.\par

\smallskip
{\bf Sub-case 1.2.1.} Let $H_2\equiv 0$. Then from (\ref{ebm.10}), we have $g^{(2)}-c_0g^{(1)}\equiv 0$, where $c_0=(2c(b-a)-2k a_{k-1} a(1-a_0))/k^2(k+1)a(1-a_0)\neq 0$. 
On integration, we get 
\bea\label{ebmm.1}\label{ebmm.2} g^{(1)}(z)=d_0\exp c_0 z\;\;\text{and}\;\;g(z)=A_0\exp c_0 z+B_0,\eea 
where $d_0\in\mathbb{C}\backslash \{0\}$, $A_0=d_0/c_0$ and $B_0\in\mathbb{C}$.
Since $a$ is not a Picard exceptional value of $f$, we have $B_0\neq 0$. Let $g(z_2)=0$. Then $f(z_2)=a$.
Now from (\ref{ebm.8}) and (\ref{ebmm.1}), we have 
$A_0^k\exp\left(k c_0 z_{2}\right)=(-B_0)^k$ and $A_0^k\exp\left(k c_0 z_{2}\right)=a(1-a_0)/k!$ from which we have
\bea\label{ebmm.5} (-B_0)^k=a(1-a_0)/k!c_0^k.\eea

Also from (\ref{ebmm.2}), we get
\bea\label{rss1} g^{k}(z)=A_0^{k}\exp kc_0z+\;{}^{k}C_{1}A_0^{k-1}B_0\exp (k-1)c_0z+\ldots+\;{}^{k}C_{k-1}A_0B_0^{k-1}\exp c_0z+B_0^{k}\eea
and so
\bea\label{rss2} (g^k(z))^{(i)}&=&k^iA_0^kc_0^i\exp kc_0z+\;{}^{k}C_{1}(k-1)^iA_0^{k-1}B_0c_0^i\exp (k-1)c_0z+\ldots\\&&+\;{}^{k}C_{k-1}A_0B_0^{k-1}c^i_0\exp c_0z\nonumber\eea
for $i\in\mathbb{N}$ and 
\bea\label{ebmm.5a} \mathscr{L}_k(g^k(z))&=&A^k_0\sideset{}{_{i=0}^{k}}{\sum}a_i(kc_0)^i\exp kc_0z+\;{}^{k}C_{1}A_0^{k-1}B_0\sideset{}{_{i=0}^{k}}{\sum}a_i((k-1)c_o)^i\exp (k-1)c_0z\nonumber\\
&&+\cdots+\;{}^{k}C_{k-1}A_0B_0^{k-1}\sideset{}{_{i=0}^{k}}{\sum}a_ic_0^i\exp c_0z+a_0B_0^k.\eea

Now using (\ref{rss1})-(\ref{ebmm.5a}) to (\ref{ebm.4}), we get 
\bea\label{ebmm.6} &&\left(kc_0-kc_0\sideset{}{_{i=0}^{k}}{\sum}a_i(kc_0)^i-c\right)A_0^{k+1}\exp (k+1)c_0z\\
&&+\left(k\;{}^{k}C_{1}c_0-k\;{}^{k}C_{1}c_0\sideset{}{_{i=0}^{k}}{\sum}a_i((k-1)c_0)^i-c\;{}^{k+1}C_{1}\right)A_0^kB_0\exp kc_0z+\cdots\nonumber\\
 &&+\left(k\; {}^{k}C_{k-1}c_0-k\; {}^{k}C_{k-1}c_0\sideset{}{_{i=0}^{k}}{\sum}a_ic_0^i-c\; {}^{k+1}C_{k-1}\right)A_0^2B_0^{k-1}\exp 2c_0z\nonumber\\
&&+\left(kc_0-c\; {}^{k+1}C_{k}\right)A_0B_0^k\exp c_0z-cB_0^{k+1}\nonumber\\
&&=A_0(c(a-b)-ka(1-a_0)c_0)\exp c_0z+c(a-b)B_0\nonumber\eea
which shows by Lemma \ref{l2} that $kc_0-kc_0\sideset{}{_{i=0}^{k}}{\sum}a_i(kc_0)^i-c=0$, $B_0^k=b-a$ and  
\bea\label{ebmm.7} A_0B_0^k\left(kc_0-c\; {}^{k+1}C_{k}\right)=A_0(c(a-b)-ka(1-a_0)c_0),\;\;\text{i.e.},\;\;c_0=(b-a)c/(b-aa_0).\eea

Since $B_0^k=b-a$, from (\ref{ebmm.5}) and (\ref{ebmm.7}), we get
\bea\label{ebmm.8}c^k=(-1)^ka(1-a_0)(b-aa_0)^k/k!(b-a)^{k+1}.\eea

Again since $c_0=(2c(b-a)-2k a_{k-1} a(1-a_0))/k^2(k+1)a(1-a_0)$, from (\ref{ebmm.7}), we get
\bea\label{ebmm.8a} (b-a)\left(2(b-aa_0)-k^2(k+1)a(1-a_0)\right)c=2ka(1-a_0)(b-aa_0)a_{k-1}.\eea

\smallskip
First suppose $(a_0, a_1,\ldots,a_{k-1})=(0,0,\ldots,0)$. Since $c\neq 0$, from (\ref{ebmm.8a}), we get $b=(k^2(k+1)a)/2$. Again since $a\neq b$, we have $k\geq 2$. Let $k=2$. Then $b=6a$, $B_0^2=5a$ and so from (\ref{ebmm.8}), we have
$c^2=18/125$. Note that $f-a=g^k$ and so from (\ref{ebmm.2}), we have 
\beas f(z)-b=10d_0^2a_2\exp(5cz/3)+(12B_0d_0/5c)\exp(5cz/6)\eeas
and 
\beas\mathscr{L}_k(f(z))-b=4d_0^2\exp(5cz/3)+(6B_0d_0/25c)\exp(5cz/6)-6a.\eeas

It is easy to verify that $f-b$ and $\mathscr{L}_k(f)-b$ have no common zeros and so $f=b\not\Rightarrow  \mathscr{L}_k(f)=b$, which is impossible here. Next suppose $k\geq 3$. Now from (\ref{ebmm.7}), we calculate that
\beas A_1:\; kc_0-kc_0\sideset{}{_{i=0}^k}{\sum}a_i(kc_0)^i-c=k(b-a)c/b-k^{k+1}((b-a)c/b)^{k+1}-c,\eeas
\beas A_2: &&k{}^{k}C_{1}c_0-k{}^{k}C_{1}c_0\sideset{}{_{i=0}^k}{\sum}a_i((k-1)c_0)^i- {}^{k+1}C_{1}c\\&=&k^2(b-a)c/b-k^2(k-1)^k((b-a)c/b)^{k+1}-(k+1)c\eeas
and
\beas A_3: k {}^{k}C_{k-1}c_0-k {}^{k}C_{k-1}c_0\sideset{}{_{i=0}^{k}}{\sum}a_ic_0^i-{}^{k+1}C_{k-1}c
&=&k^2(b-a)c/b-k^2((b-a)c/b)^{k+1}\\&&-k(k+1)c/2.\eeas

If possible suppose $A_i=0$, for $i=1,2,3$. 
Now $A_1=0$ and $A_2=0$ imply that
\bea\label{ss4} (b-a)k^2\left(1-((k-1)/k)^k\right)=b\left((k+1)-k((k-1)/k)^k\right).\eea

Again $A_1=0$ and $A_3=0$ imply that
\bea\label{ss5}  (b-a)k^2(1-1/k^k)=b\left(k(k+1)/2-1/k^{k-1}\right).\eea

Now from (\ref{ss4}) and (\ref{ss5}), we get $k(k-1)^{k+1}-(k-2)(k+1)k^k=2$ which is impossible for $k\geq 3$. 
Hence $A_1=0$, $A_2=0$ and $A_3=0$ can not hold simultaneously. Therefore using Lemma \ref{l2} to (\ref{ebmm.6}), we get a contradiction.

\smallskip
Next suppose $(a_0,a_1,\ldots,a_{k-1})\not=(0,0,\ldots,0)$. In this case 
\[f(z)=\left(A_0\exp c_0 z+B_0\right)^k+a,\]
where $A_0\in\mathbb{C}\backslash \{0\}$, $B_0^k=b-a$, $c_0=(b-a)c/(b-aa_0)$ and $\varphi=c$ such that \[c^k=(-1)^ka(1-a_0)(b-aa_0)^k/k!(b-a)^{k+1}\;\;\text{and}\;\;kc_0-kc_0\sideset{}{_{i=0}^{k}}{\sum}a_i(kc_0)^i-c=0.\]

\smallskip
{\bf Sub-case 1.2.2.} Let $H_2\not\equiv 0$. Since $g$ has only simple zeros and 
\[g=0\Rightarrow k^2(k+1)a(1-a_0)g^{(2)}+2\left(k a_{k-1} a(1-a_0)+c(a-b)\right)g^{(1)}=0,\]
from (\ref{ebm.10}), we see that $H$ is an entire function.
Using Lemma \ref{l1} to (\ref{ebm.10}), we can say that $H_2$ is a constant, say $\lambda_0\in\mathbb{C}\backslash \{0\}$ and so
\bea\label{ebm.12} g^{(2)}=\check{a}_1g^{(1)}+\check{b}_1g,\;\;
\text{where}\;\; \check{a}_1=c_0\;\;\text{and}\;\;\check{b}_1=\lambda_0/k^2(k+1)a(1-a_0)\neq 0.\eea 

Differentiating (\ref{ebm.12}) and using it repeatedly, we have
\bea\label{ebm.13} g^{(i)}=\check{a}_{i-1}g^{(1)}+\check{b}_{i-1}g,\eea
where $i\geq 2$ and $\check{a}_{i-1}, \check{b}_{i-1}\in\mathbb{C}$. Let $\lambda_1, \lambda_2$ be the roots of $m^2-\check a_1m-\check b_1=0$. Then
\bea\label{ebm.17} \lambda_1+\lambda_2=c_0\;\;\text{and}\;\;\lambda_1 \lambda_2=-\lambda_0/k^2(k+1)a(1-a_0).\eea

Clearly $\lambda_1\neq 0$ and $\lambda_2\neq 0$. Now the equation (\ref{ebm.12}) has one of the following solutions:
\begin{enumerate}
\item[(1)] $g(z)=(A_1 z+B_1)\exp \lambda_1z$, where $A_1\in\mathbb{C}\backslash \{0\}$, $B_1\in\mathbb{C}$, if $\lambda_1=\lambda_2$;

\smallskip
\item[(2)] $g(z)=A_2\exp \lambda_1 z+B_2\exp \lambda_2 z$, where
where $A_2, B_2\in\mathbb{C}\backslash \{0\}$, if $\lambda_1\neq \lambda_2$.
\end{enumerate}

Now we divide following two sub-cases.\par

\smallskip
{\bf Sub-case 1.2.2.1.} Let $g(z)=(A_1 z+B_1)\exp \lambda_1z$, where $A_1\in\mathbb{C}\backslash \{0\}$, $B_1\in\mathbb{C}$. Then
$f(z)=P(z)\exp k\lambda_1 z+a$, where $P(z)=(A_1 z+B_1)^k$ and so
$f^{(i)}(z)=P_i(z)\exp \lambda_1 z$, where 
\[P_i=(\exp (k\lambda_1)^i P+\sideset{}{_{j=1}^i}{\sum} {}^{i}C_{j} (\exp k\lambda_1)^{i-j}P^{(j)}.\]

Therefore $\mathscr{L}_k(f(z))=Q(z)\exp(k\lambda_1 z)+aa_0$, where 
\[Q(z)=\sideset{}{_{i=0}^k}{\sum}a_i(k\lambda_1)^i\;P+\sideset{}{_{i=1}^{k}}{\sum}a_i\sideset{}{_{j=1}^i}{\sum} {}^{i}C_{j} (k\lambda_1)^{i-j}P^{(j)}.\]

Clearly $f-b$ has infinitely many zeros. Let $z_3$ be a zero of $f-b$. Then $f(z_3)=b$. Since $f=b\Rightarrow \mathscr{L}_k(f)=b$, we have $\mathscr{L}_k(f(z_3))=b$. Therefore 
$P(z_3)\exp (k\lambda_1) z_3=b-a$ and $Q(z_3)\exp((k\lambda_1) z_3)=b-aa_0$. If $b=aa_0$, then $Q(z_3)=0$ and so $f-b$ has only finitely many zeros, which is impossible. Hence $b\neq aa_0$. Eliminating $\exp (k\lambda_1) z_3$, we get $Q(z_3)/P(z_3)=b/(b-aa_0)$, from which we get $Q/P\equiv b/(b-aa_0)$ and so from above, we have 
\beas\label{e3.1.11} b/(b-aa_0)=Q/P=\sideset{}{_{i=0}^k}{\sum} a_i (k\lambda_1)^i+\sideset{}{_{i=1}^{k}}{\sum} a_i\sideset{}{_{j=1}^i}{\sum} {}^{i}C_{j} (k\lambda_1)^{i-j}P^{(j)}/P,\eeas
which shows that $P$ is a constant, which is absurd.

\smallskip
{\bf Sub-case 1.2.2.2.} Let $g(z)=A_2\exp \lambda_1 z+B_2\exp \lambda_2 z$, where $A_2, B_2\in\mathbb{C}\backslash \{0\}$. Clearly $g^{(1)}(z)=A_2\lambda_1 \exp \lambda_1 z+B_2\lambda_2 \exp \lambda_2 z$. Also from (\ref{ebm.17}), we have
\bea\label{ebm2.3a} \lambda_1+\lambda_2=(2c(b-a)-2ka(1-a_0)a_{k-1})/k^2(k+1)a(1-a_0).\eea

Now from (\ref{ebm.4}), we get $N(r,0;g^{(1)}\mid g^k\neq b-a)=0$. If $f-b=g^k+a-b$ has no multiple zeros, then $g^{(1)}\neq 0$, which is absurd. Hence $N(r,b-a;g^k\mid \geq 2)\neq 0$. Let $d=\lambda_1-\lambda_2$. Then
\bea\label{ebm.22} g^i(z)=\left(A_2^i\exp id z+\;{}^{i}C_{1}A_2^{i-1}B_2\exp (i-1)d z+\ldots+B_2^i\right)\exp i\lambda_2 z\eea
and 
\bea\label{ebm.23} (g^{(1)}(z))^i=\left(\lambda_1^i A_2^i\exp id z+\;{}^{i}C_{1}\lambda_1^{i-1}A_2^{i-1}\lambda_2 B_2 \exp (i-1)d z+\ldots+\lambda_2^iB_2^i\right)\exp i\lambda_2 z\eea
for $i=1,\ldots, k+1$. Let $z_2$ be a zero of $g$. From (\ref{ebm.8}), we get 
\[(g^{(1)}(z_{2}))^k=a(1-a_0)/k!.\]

Since $g(z_2)=0$, we have 
\[A_2\exp \lambda_1 z_2+B_2\exp \lambda_2 z_2=0,\;\;\text{i.e.,}\;\;A_2\exp \lambda_1 z_2=-B_2\exp \lambda_2 z_2.\]

Again since $(g^{(1)}(z_{2}))^k=a(1-a_0)/k!$, we get 
\[(A_2\lambda_1\exp \lambda_1 z_1+B_2\lambda_2\exp \lambda_2 z_2)^k=a(1-a_0)/k!\]
and so $\exp k\lambda_2 z_2=a(1-a_0)/k!B_2^k (\lambda_2-\lambda_1)^k$. Also 
\[\exp k\lambda_1 z_2=((-1)^kB_2^k/A_2^k)\exp k\lambda_2 z_2=a(1-a_0)/k!A_2^k (\lambda_1-\lambda_2)^k.\]

Consequently $A_2$ and $B_2$ are connected by the relation 
\bea\label{aaa.1} \left(a(1-a_0)/k!B_2^k (\lambda_2-\lambda_1)^k\right)^{\lambda_1}=\left(a(1-a_0)/k!A_2^k (\lambda_1-\lambda_2)^k\right)^{\lambda_2}.\eea

Now from (\ref{ebm.2}) and (\ref{ebm.13}), we have
\bea\label{ebm.20} (g^k)^{(k)}=\hat{c}_0(g^{(1)})^k+\hat{c}_1g(g^{(1)})^{k-1}+\ldots+\hat{c}_{k-1}g^{k-1}g^{(1)}+\hat{c}_kg^k,\eea
where $\hat c_0=k!$ and $\hat{c}_i\in\mathbb{C}$ for $i\geq 2$. Differentiating (\ref{ebm.20}) once, we get
\beas (g^k)^{(k+1)}&=& \hat{c}_0k(g^{(1)})^{k-1}g^{(2)}+\hat{c}_1(g^{(1)})^k+\hat{c_1}(k-1)g (g^{(1)})^{k-2}g^{(2)}\\&&+\ldots+ \hat {c}_{k-1}(k-1)g^{k-2}(g^{(1)})^2+\hat{c}_{k-1}g^{k-1}g^{(2)}+\hat{c}_kkg^{k-1}g^{(1)}\eeas
and so using (\ref{ebm.12}), we get
\bea\label{ebm.20a} (g^k)^{(k+1)}&=& (k\check{a}_1\hat{c}_0+\check{c}_1)(g^{(1)})^k+(k\check{b}_1\hat{c}_0+(k-1)\check{a}_1\hat {c}_1)g(g^{(1)})^{k-1}\\
&&+\ldots+(\check{a}_1\hat{c}_{k-1}+k\hat{c}_k)g^{k-1}g^{(1)}+\check{b}_1\hat{c}_{k-1}g^k.\nonumber\eea

Let $z_2$ be a zero of $g$. Then from (\ref{ebm.8}), we have $(g^{(1)}(z_{2}))^k=a(1-a_0)/k!$. Consequently
from (\ref{ebm1.2}) and (\ref{ebm.20a}), we get
\[\hat{c}_1=k!(k-1)c(b-a)/k(k+1)a(1-a_0)\;-k!(k-1)a_{k-1}/(k+1).\]

Also using (\ref{ebm.20}) to (\ref{ebm.4}), we get 
\bea\label{ebm.21} \sideset{}{_{i=0}^{k+1}}{\sum}a_{1i}g^i(g^{(1)})^{k+1-i}=b_{10}g+b_{11}g^{(1)},\eea
where $a_{1i}\in\mathbb{C}$ such that 
\bea\label{al1} a_{10}=k!,\;a_{11}=\hat{c}_1+k!a_{k-1},\; b_{10}=-c(a-b)/k\;\text{and}\;b_{11}=a(1-a_0).\eea

Let $z_3$ be a multiple zero of $f-b$. Then $g^{(1)}(z_3)=0$, $g^k(z_3)=b-a$ and so from (\ref{ebm.21}), we get $(b-a)a_{1,k+1}=b_{10}$, i.e., $a_{1,k+1}=-c/k$. Using (\ref{ebm.12}), (\ref{ebm.20}) and (\ref{ebm.20a}) to (\ref{ebm.5}), we get
\bea\label{ebm.21a} \sideset{}{_{i=0}^{k+1}}{\sum}a_{2i}g^i(g^{(1)})^{k+1-i}=b_{20}g+b_{21}g^{(1)},\eea
where $a_{2i}\in\mathbb{C}$ such that
\bea\label{al2} a_{20}
=(k!(k^2+k+2)(b-a)c)/(k(k+1)a(1-a_0))-2k!a_{k-1}/(k+1),\eea
\bea\label{al3} a_{21}&=&
\frac{k!}{k(k+1)a(1-a_0)}\left((k+1)\lambda_0+\frac{2(k-1)(b-a)^2c^2}{(k+1)a(1-a_0)}\right)\\
&&+\frac{k!(k^3-4k^2+5k+2)}{k(k+1)^2}\frac{(b-a)c}{a(1-a_0)}a_{k-1}
-\frac{k!(k^3-2k^2+3k+2)}{(k+1)^2}a_{k-1}^2+kk!a_{k-2}\nonumber\eea
and
\bea\label{al4} b_{20}=ka(1-a_0)\check{b}_1\;\;\text{and}\;\;b_{21}=ka(1-a_0)\check{a}_1-(a-b)c.\eea

If $z_3$ is a multiple zero of $f-b$, then by a simple calculation on (\ref{ebm.21a}), we get 
\[(b-a)a_{2, k+1}=b_{20},\;\;\text{i.e.,}\;\;a_{2, k+1}=ka(1-a_0)\check b_1/(b-a)\neq 0.\]

Now we consider following two sub-cases.\par

\smallskip
{\bf Sub-case 1.2.2.2.1.} Let $a_{20}=0$. If possible suppose $b_{21}\neq 0$. Now if $z_3$ is a multiple zero of $f-b$, then by a simple calculation on (\ref{ebm.21a}), we get a contradiction. Hence $b_{21}=0$ and so from (\ref{ebm.21a}), we have 
\[\sideset{}{_{i=0}^k}{\sum} a_{2,i+1}g^i(g^{(1)})^{k-i}=b_{20}.\]

Using this, we get from (\ref{ebm.21}) that
\bea\label{bbb.1}\sideset{}{_{i=0}^k}{\sum}\left(a_{1i}-b_{11}a_{2,i+1}/b_{20}\right)g^i(g^{(1)})^{k+1-i}+a_{1,k+1}g^{k+1}=b_{10}g,\eea
which demands that $a_{10}b_{20}-a_{21}b_{11}=0$ and so from (\ref{al1}), (\ref{al3}) and (\ref{al4}), we get
\bea\label{aaa.3} \lambda_0&=&-\frac{2(k-1)(b-a)^2c^2}{k(k+1)a(1-a_0)}-\frac{k^3-4k^2+5k+2}{k(k+1)}c(b-a)a_{k-1}+\frac{k^3-2k^2+3k+2}{(k+1)}\times \nonumber\\&&a(1-a_0)a_{k-1}^2-k(k+1)a(1-a_0)a_{k-2}.\eea

Clearly from (\ref{ebm.17}) and (\ref{aaa.3}), we get
\bea\label{aaa.4} \lambda_1\lambda_2&=&\frac{2(k-1)(b-a)^2c^2}{k^3(k+1)^2a^2(1-a_0)^2}+\frac{k^3-4k^2+5k+2}{k^3(k+1)^2}\frac{c(b-a)a_{k-1}}{a(1-a_0)}-\frac{k^3-2k^2+3k+2}{k^2(k+1)^2}\times\nonumber\\&&a_{k-1}^2+a_{k-2}/k.\eea

Now from (\ref{ebm.22}), (\ref{ebm.23}) and (\ref{bbb.1}), we get
\bea\label{e3.33ss}Q_{k}\left(\exp dz\right)\exp k\lambda_2 z=b_{10},\eea
where 
\[Q_{k}\left(\exp dz\right)=d_0\exp k dz+d_1\exp (k-1)dz+\ldots+d_{k+1}\]
and $d_i\in\mathbb{C}$, $i=0,1,2,\ldots, k+1$. 
Therefore from (\ref{e3.33ss}), we get $Q_{k}\left(\exp dz\right)=C_0\exp mdz$, where $C_0\in\mathbb{C}\backslash \{0\}$ and $0\leq m\leq k$.  Consequently from (\ref{e3.33ss}), we have $C_0\exp (k\lambda_2+md) z=b_{10}$ and so $k\lambda_2+md=0$. Since $\lambda_1, \lambda_2\neq 0$, we get $m\neq 0, k$ and so $1\leq m\leq k-1$, which demands that $k\geq 2$. Again since $d=\lambda_1-\lambda_2$, we have $\lambda_1/\lambda_2=(m-k)/m$ and so from (\ref{ebm2.3a}), we have
\bea\label{e3.37}\begin{cases}
\lambda_1=\frac{(m-k)(2(b-a)c-2ka(1-a_0)a_{k-1})}{k^2(k+1)(2m-k)a(1-a_0)},\\

\medskip
\lambda_2=\frac{m(2(b-a)c-2ka(1-a_0)a_{k-1})}{k^2(k+1)(2m-k)a(1-a_0)},
\end{cases}\eea
where $m\in\mathbb{N}$ such that $1\leq m\leq k-1$. Putting the values of $\lambda_1$ and $\lambda_2$ in (\ref{aaa.4}), we get
\bea\label{e3.38s}&&\left(\frac{4m(m-k)}{(2m-k)^2}-2k(k-1)\right)(b-a)^2c^2-k\left(\frac{8m(m-k)}{(2m-k)^2}+k^3-2k^2+3k+2\right)\times\nonumber\\&& a(b-a)a_{k-1}c+k^2\left(4k\frac{m(m-k)}{(2m-k)^2}+k^3-2k^2+3k+2\right)a^2(1-a_0)a_{k-1}^2\nonumber\\&&-k^3(k+1)^2 a^2(1-a_0)^2a_{k-2}=0.\eea

\smallskip
{\bf Sub-case 1.2.2.2.2.} Let $a_{20}\neq 0$. 

\smallskip
First suppose $(b_{10},b_{11})=K_1(b_{20}, b_{21})$, where $K_1\in\mathbb{C}\backslash \{0\}$. Then by a routine calculation, we get from (\ref{al1}) and (\ref{al4}) that
\bea\label{rob2.0} \lambda_0=\frac{k^2+k+2}{k}\frac{(b-a)^2c^2}{a(1-a_0)}-2(b-a)a_{k-1}c.\eea

Clearly from (\ref{ebm.17}) and (\ref{rob2.0}), we get
\bea\label{rob2.1} \lambda_1\lambda_2=-\frac{(k^2+k+2)}{k^3(k+1)}\frac{c^2(b-a)^2}{a^2(1-a_0)^2}+\frac{2(b-a)a_{k-1}c}{k^2(k+1)a(1-a_0)}.\eea

Now from (\ref{ebm.22}), (\ref{ebm.23}) and (\ref{ebm.21}), we get
\bea\label{e3.33} \hat{Q}_{k+1}\left(\exp dz\right)\exp k\lambda_2 z=d_0\exp d z+d_1,\eea
where 
\[\hat{Q}_{k+1}\left(\exp dz\right)=d_{10}\exp (k+1) dz+d_{11}\exp kdz+\ldots+d_{1,k+1}\]
and $d_0, d_1, d_{1i}\in\mathbb{C}$, $i=0,1,2,\ldots, k+1$. Then from (\ref{e3.33}), we get 
\[\hat{Q}_{k+1}\left(\exp dz\right)=C_1\exp mdz (d_0\exp d z+d_1),\]
where $C_1\in\mathbb{C}\backslash \{0\}$ and $0\leq m\leq k$.  So from (\ref{e3.33}), we have $C_1\exp (k\lambda_2+md) z=1$. Now proceeding in the same way as done above, we get the same values of $\lambda_1$ and $\lambda_2$ as given by (\ref{e3.37}). In this case also $k\geq 2$.
Putting the values of $\lambda_1$ and $\lambda_2$ from (\ref{e3.37}) into (\ref{rob2.1}), we get
\bea\label{e3.38}&&\left(\frac{4m(m-k)}{(2m-k)^2}+k(k+1)(k^2+k+2)\right)(b-a)^2c^2\\&&-k\left(\frac{8m(m-k)}{(2m-k)^2}+2k(k+1)\right)a(b-a)a_{k-1}c+\frac{m(m-k)4k^3}{(2m-k)^2}a^2(1-a_0)^2a_{k-1}^2=0.\nonumber\eea

\smallskip
Next we suppose $(b_{10},b_{11})\neq K_1(b_{20}, b_{21})$, where $K_1\in\mathbb{C}\backslash \{0\}$. If $(a_{10}, a_{11}, \ldots, a_{1,k+1})=K_2(a_{20}, a_{21}, \ldots, a_{2,k+1})$, where $K_2\in\mathbb{C}\backslash \{0\}$. Then from (\ref{ebm.21}) and (\ref{ebm.21a}), we get 
\[(b_{10}-K_2 b_{20})g+(b_{11}-K_2b_{21})g^{(1)}=0.\]

Since $(b_{10},b_{11})\neq K_1(b_{20}, b_{21})$, at least one of $b_{10}-K_2 b_{20}$ and $b_{11}-K_2b_{21}$ is non-zero and so we get a contradiction. Hence $(a_{10}, a_{11}, \ldots, a_{1,k+1})\neq K_2(a_{20}, a_{21}, \ldots, a_{2,k+1})$, where $K_2\in\mathbb{C}\backslash \{0\}$. Suppose $K_2=a_{10}/a_{20}$. Clearly $(a_{11}, \ldots, a_{1,k+1})\neq \frac{a_{10}}{a_{20}}(a_{21}, \ldots, a_{2,k+1})$.
Multiplying (\ref{ebm.21}) by $a_{20}$ and (\ref{ebm.21a}) by $a_{10}$ and then subtracting, we get
\bea\label{ebm2.1} &&D_1(g^{(1)})^k+D_2g(g^{(1)})^{k-1}+\ldots+D_{k-1}g^{k-2}(g^{(1)})^2+D_kg^{k-1}g^{(1)}+D_{k+1} g^k\\&=&(b-a)D_{k+2},\nonumber\eea
where $D_i\in\mathbb{C}$ for $i=1,2,\ldots,k+1$ such that
\beas D_1&=&(k!)^2\bigg(\frac{\lambda_0}{ka(1-a_0)}-\frac{(k-1)(k^2-k+2)}{k^2(k+1)^2}\frac{(b-a)^2c^2}{a^2(1-a_0)^2}+\frac{k^3-6k^2+3k-4}{k(k+1)^2}\times\\
 &&\;\;\;\;\frac{(b-a)a_{k-1}c}{a(1-a_0)}-\frac{k^3-2k^2+3k-2}{(k+1)^2}a_{k-1}^2+ka_{k-2}\bigg)\eeas
and
\beas D_{k+2}=k!\left(\frac{\lambda_0}{k(k+1)(b-a)}-\frac{k^2+k+2}{k^2(k+1)}\frac{(b-a)c^2}{a(1-a_0)}+\frac{2a_{k-1}}{(k+1)k}\right).\eeas

If $D_1=0$ and $D_{k+2}\neq 0$, then from (\ref{ebm2.1}), we get $N(r,0;g)=0$, which is impossible. Similarly if $D_1\neq 0$ and $D_{k+2}=0$, then from (\ref{ebm2.1}), we again get a contradiction. Hence either $D_1=0$ and $D_{k+2}=0$ or $D_1\neq 0$ and $D_{k+2}\neq 0$.

First suppose $D_1=0$ and $D_{k+2}=0$. Since $(a_{11}, \ldots, a_{1,k+1})\neq \frac{a_{10}}{a_{20}}(a_{21}, \ldots, a_{2,k+1})$, it follows that at least one of $D_2$, $\ldots$, $D_{k+1}$ is non-zero. Then we get a contradiction.

Next suppose $D_1\neq 0$ and $D_{k+2}\neq 0$. Let $z_{2}$ be a zero of $g$. Then from (\ref{ebm.8}), we have $(g^{(1)}(z_{2}))^k=a(1-a_0)/k!$. Consequently from (\ref{ebm2.1}), we get $a(1-a_0)D_1/k!=(b-a)D_{k+2}$ and so
\bea\label{ebm2.2} \lambda_0&=&-\frac{4(k^2+1)}{k^2(k+1)}\frac{(b-a)^2c^2}{a(1-a_0)}-\frac{(k^3-6k^2+5k-4)}{k(k+1)}(b-a)ca_{k-1}\\
&&+\frac{k^3-2k^2+3k-2}{(k+1)}a(1-a_0)a_{k-1}^2+\frac{2(b-a)a_{k-1}}{k}-k(k+1)a(1-a_0)a_{k-2}.\nonumber\eea
Therefore from (\ref{ebm.17}) and (\ref{ebm2.2}), we have
\bea\label{ebm2.3} \lambda_1\lambda_2&=&\frac{4(k^2+1)}{k^4(k+1)^2}\frac{(b-a)^2c^2}{a^2(1-a_0)^2}+\frac{k^3-6k^2+5k-4}{k^2(k+1)^2}\frac{(b-a)c}{a(1-a_0)}a_{k-1}\nonumber\\
&&-\frac{k^3-2k^2+3k-2}{k^2(k+1)^2}a_{k-1}^2-\frac{2(b-a)a_{k-1}}{k^3(k+1)a(1-a_0)}+\frac{a_{k-2}}{k}.\eea

Proceeding in the same way as done above, we get the same values of $\lambda_1$ and $\lambda_2$ as given by (\ref{e3.37}). In this case  $k\geq 2$. Putting the values of $\lambda_1$ and $\lambda_2$ from (\ref{e3.37}) into (\ref{ebm2.3}), we get
\bea\label{e3.38ss}&&4\left(\frac{m(m-k)}{(2m-k)^2}-k^2-1\right)(b-a)^2c^2-k\left(\frac{8m(m-k)}{(2m-k)^2}+k^3-6k^2+5k-4\right)a(b-a)a_{k-1}c\nonumber\\&&+k^2\left(4k\frac{m(m-k)}{(2m-k)^2}+k^3-2k^2+3k-2\right)a^2(1-a_0)a_{k-1}^2
+2k(k+1)a(b-a)a_{k-1}\nonumber\\&&-k^3(k+1)^2 a^2(1-a_0)^2a_{k-2}=0.\eea

\smallskip
If $(a_{k-2}, a_{k-1})=(0,0)$, then from (\ref{e3.38s}), (\ref{e3.38}) and (\ref{e3.38ss}), we get a contradiction, since $m-k<0$ and $k\geq 2$. So $(a_{k-2},a_{k-1})\neq (0,0)$. Therefore
\[f(z)=\left(A_2\exp \lambda_1 z+B_2\exp \lambda_2 z \right)^k+a,\]
where $A_2, B_2\in\mathbb{C}\backslash \{0\}$ satisfy the relation given by (\ref{aaa.1}), $\lambda_1, \lambda_2$ are given by (\ref{e3.37}) and $\varphi=c$ satisfies one of the equations given by (\ref{e3.38s}), (\ref{e3.38}) and (\ref{e3.38ss}).\par

\smallskip
{\bf Case 2.} Let $\varphi\equiv 0$. Since $f^{(1)}\not\equiv 0$, it follows that $\mathscr{L}_k(f)\equiv f$, i.e.,
\bea\label{k1} f^{(k)}+a_{k-1}f^{(k-1)}+\ldots+a_1f^{(1)}+(a_0-1)f\equiv 0.\eea

Now by Lemma \ref{l2.4}, we have $\rho(f)\leq 1$. If $(a_0,a_1, \ldots, a_{k-1})=(0,0,\ldots,0)$, then the conclusions of Theorem \ref{t1} follow from Lemma \ref{l2.5}. Next suppose $(a_0, a_1,\ldots, a_{k-1})\neq (0,0,\ldots,0)$. Now differentiating (\ref{k1}), we have 
\bea\label{k3} f^{(k+1)}+a_{k-1}f^{(k)}+\ldots+a_1f^{(2)}+(a_0-1)f^{(1)}\equiv 0.\eea

\smallskip
First suppose $a_0=1$. Since all the zeros of $f-a$ have multiplicity at least $k$, from (\ref{k1}), we get that $f-a$ has no zeros. So $f(z)=A\exp \lambda z+a$, where $A, \lambda\in\mathbb{C}\backslash \{0\}$. Also from (\ref{k1}), we get $\lambda^{k-1}+a_{k-1}\lambda^{k-2}+\ldots+a_1=0$.

\smallskip
Next suppose $a_0\neq 1$. Now we consider following sub-cases.\par

\smallskip
{\bf Sub-case 2.1.} Let $a=0$. In this case also $f$ has no zeros and so 
\[f(z)=A\exp \lambda z,\]
where $A, \lambda\in\mathbb{C}\backslash \{0\}$ such that $\lambda^k+a_{k-1}\lambda^{k-1}+\ldots+a_1\lambda+a_0=1$.\par

\smallskip
{\bf Sub-case 2.2.} Let $a\neq 0$. Now from (\ref{k1}), we see that all the zeros of $f-a$ have multiplicity exactly $k$. If $a$ is a Picard exceptional value of $f$, then 
\[f(z)=A\exp \lambda z+a,\]
where $A, \lambda\in\mathbb{C}\backslash \{0\}$ such that $\lambda^k+a_{k-1}\lambda^{k-1}+\ldots+(a_0-1)=0$. Henceforth we assume that $a$ is not a Picard exceptional value of $f$. Therefore we assume that $f-a=g^k$, where $g$ has only simple zeros and $\rho(g)=1$.
Using (\ref{ebm.2})-(\ref{ebm.2a}) to (\ref{k1}), we get
\bea\label{k2}k!(g^{(1)})^k+R_{10}(g)+(a_0-1)(g^k+a)\equiv 0,\eea
where $R_{10}(g)$ is a differential polynomial. Again using (\ref{ebm.2})-(\ref{ebm.3}) to (\ref{k3}), we get
\bea\label{k4}k!\left(k(k+1)g^{(2)}/2+a_{k-1}g^{(1)}\right)(g^{(1)})^{k-1}+R_{11}(g)\equiv 0,\eea
where $R_{11}(g)$ is a differential polynomial. Let $z_2$ be a zero of $g$. Clearly $g^{(1)}(z_2)\neq 0$. Therefore from (\ref{k2}), we have
\bea\label{k5} (g^{(1)}(z_2))^k=a(1-a_0)/k!.\eea

Now (\ref{k4}) yields $k(k+1)g^{(2)}(z_2)+2a_{k-1}g^{(1)}(z_2)=0$ and so 
\[g=0\Rightarrow k(k+1)g^{(2)}+2a_{k-1}g^{(1)}=0.\]

Let 
\bea\label{k6} H_3=(k(k+1)g^{(2)}+2a_{k-1}g^{(1)})/g.\eea

Also from (\ref{k1}), we see that
\bea\label{k7}\sideset{}{_{i=0}^k}{\sum}a_i(g^k)^{(i)}=a(1-a_0)+g^k.\eea

Now we consider following sub-cases.\par

\smallskip
{\bf Sub-case 2.2.1.} Let $H_3\equiv 0$. Then $k(k+1)g^{(2)}+2a_{k-1}g^{(1)}\equiv 0$. Note that $g^{(2)}\not\equiv 0$. If $a_{k-1}=0$, then we get a contradiction. Hence $a_{k-1}\neq 0$. On integration, we get
\bea\label{k8} g^{(1)}(z)=d_0\exp c_0 z\;\;\text{and}\;\;g(z)=A_0\exp c_0 z+B_0,\eea 
where $d_0\in\mathbb{C}\backslash \{0\}$, $c_0=-2a_{k-1}/k(k+1)$, $A_0=d_0/c_0$ and $B_0\in\mathbb{C}\backslash \{0\}$. If $z_2$ is a zero of $g$, then by routine calculations, we get from (\ref{k5}) and (\ref{k8}) that $(-B_0)^k=a(1-a_0)/k!c_0^k$.
Now from (\ref{ebmm.5a}) and (\ref{k7}), we get
\beas &&A^k_0\left(\sideset{}{_{i=0}^{k}}{\sum}a_i(kc_0)^i-1\right)\exp kc_0z\\&&+\;{}^{k}C_{1}A_0^{k-1}B_0\left(\sideset{}{_{i=0}^{k}}{\sum}a_i((k-1)c_0)^i-1\right)\exp (k-1)c_0z+\cdots\\&&+\;{}^{k}C_{k-1}A_0B_0^{k-1}\left(\sideset{}{_{i=0}^{k}}{\sum}a_ic_0^i-1\right)\exp c_0z=(1-a_0)(B_0^k-a)\nonumber\eeas
and so by Lemma \ref{l2}, we can obtain $B_0^k=a$ and $\sideset{}{_{i=0}^{k}}{\sum}a_ic_0^i=1$. Therefore
\[f(z)=(c_0\exp \lambda z+c_1)^k+a,\]
where $a,(a_0-1), c_0\in\mathbb{C}\backslash \{0\}$, $a_0,a_1,\ldots, a_{k-1})\neq (0,\ldots,0)$, $c_1^k=a$ and $\varphi=0$ such that $\lambda=-2a_{k-1}/k(k+1)$ and $\sideset{}{_{i=0}^{k}}{\sum}a_ic_0^i=1$.\par

\smallskip
{\bf Sub-case 2.2.2.} Let $H_3\not\equiv 0$. In this case, using Lemma \ref{l1} to (\ref{k6}), one can easily conclude that $H_3=\lambda_0\in\mathbb{C}\backslash \{0\}$. Now from (\ref{k6}), we get
\bea\label{k10} g^{(2)}=\check{a}_1g^{(1)}+\check{b}_1g,\;\;
\text{where}\;\; \check{a}_1=c_0\;\;\text{and}\;\;\check{b}_1=2\lambda_0/k(k+1)\neq 0.\eea

Differentiating (\ref{k10}) and using it repeatedly, we have $g^{(i)}=\check{a}_{i-1}g^{(1)}+\check{b}_{i-1}g$,
where $i\geq 2$ and $\check{a}_{i-1}, \check{b}_{i-1}\in\mathbb{C}$. Let $\lambda_1, \lambda_2$ be the roots of $m^2-\check{a}_1m-\check {b}_1=0$. Then
\bea\label{k13} \lambda_1+\lambda_2=c_0\;\;\text{and}\;\;\lambda_1 \lambda_2=-2\lambda_0/k(k+1).\eea

Consequently Eq. (\ref{k10}) has one of the following solutions:
\begin{enumerate}
\item[(1)] $g(z)=(A_1 z+B_1)\exp \lambda_1z$, where $A_1\in\mathbb{C}\backslash \{0\}$, $B_1\in\mathbb{C}$, if $\lambda_1=\lambda_2$;

\smallskip
\item[(2)] $g(z)=A_2\exp \lambda_1 z+B_2\exp \lambda_2 z$, where
where $A_2, B_2\in\mathbb{C}\backslash \{0\}$, if $\lambda_1\neq \lambda_2$.
\end{enumerate}

If $g(z)=(A_1 z+B_1)\exp \lambda_1z$, then proceeding same as done in Sub-case 1.2.2.1, we get a contradiction. Hence 
\[g(z)=A_2\exp \lambda_1 z+B_2\exp \lambda_2 z,\]
where $A_2$ and $B_2$ satisfy the relation given by (\ref{aaa.1}).
Now from (\ref{ebm.20}) and (\ref{k10}), we see that
\bea\label{k14} \hat{c}_0(g^{(1)})^{k+1}+(\hat{c}_1+a_{k-1}k!)g(g^{(1)})^k+\ldots+=a(1-a_0).\eea

If $d=\lambda_1-\lambda_2$, then from (\ref{ebm.22}), (\ref{ebm.23}) and (\ref{k14}), we get 
\[Q_{k}\left(\exp dz\right)\exp k\lambda_2 z=a(1-a_0).\]

Therefore proceeding in the same way as done in Sub-case 1.2.2.2.1, we have
$\lambda_1/\lambda_2=(m-k)/m$, where $1\leq m\leq k-1$, which demands that $k\geq 2$ and so from (\ref{k13}), we get
\[\lambda_1=-2(m-k)a_{k-1}/k(k+1)(2m-k)\;\;\text{and}\;\;\lambda_2=-2ma_{k-1}/k(k+1)(2m-k),\]
where $1\leq m\leq k-1$. Therefore
\[f(z)=\left(A_2\exp \lambda_1 z+B_2\exp \lambda_2 z \right)^k+a,\]
where $A_2, B_2\in\mathbb{C}\backslash \{0\}$ and $\lambda_1, \lambda_2$ are given above.

This completes the proof.
\end{proof}

\section{\bf{Statements and declarations}}
\vspace{1.3mm}

\noindent \textbf {Conflict of interest:} The authors declare that there are no conflicts of interest regarding the publication of this paper.\vspace{1.5mm}

\noindent{\bf Funding:} There is no funding received from any organizations for this research work.
\vspace{1.5mm}

\noindent \textbf {Data availability statement:}  Data sharing is not applicable to this article as no database were generated or analyzed during the current study.

\end{document}